%% file: main.tex
\documentclass[9pt,twocolumn,oneside]{article}
\usepackage{amsmath,amssymb,amsfonts, amsthm}
\usepackage{algorithmic}
\usepackage{graphicx}
\usepackage{textcomp}

\usepackage[utf8]{inputenc}
\usepackage[T1]{fontenc}
\usepackage{algorithm}
\usepackage{import}
\usepackage{url}
\usepackage{float} 
\usepackage[dvipsnames]{xcolor}
\usepackage{array}

\usepackage{mathtools}
\usepackage{subcaption}

\usepackage[colorlinks=true]{hyperref} 

\usepackage{xspace}

\usepackage{bm} 

\usepackage{graphicx}
\graphicspath{{./figures/}} 

\usepackage{parskip} 
\usepackage{comment}
\usepackage{fancyhdr}

\usepackage{tikz}
\usetikzlibrary{automata, positioning}
\usetikzlibrary{external}
\usepackage{csquotes}
\usepackage[backend=biber,citestyle=authoryear,bibstyle=alphabetic, backref]{biblatex}
\input{./macros/environments_macros.tex}
\input{./macros/text_macros.tex}
\input{./macros/maths_macros.tex}

\input{./macros/flow_redirection_macros.tex}
\input{./macros/experiment_macros.tex}
\input{./macros/authors_citations_macros.tex}

\begin{document}
\title{Flow Redirection for Epidemic Reaction-Diffusion Control}

\author{Pierre-Yves Massé\textsuperscript{\textsection}, Quentin Laborde\textsuperscript{\textsection}, Maria Cherifa, Jules Olayé and Laurent Oudre}

\maketitle

\begingroup\renewcommand\thefootnote{\textsection}
\footnotetext{Equal contribution}
\endgroup

\begin{abstract}
We show we can control an epidemic reaction-diffusion on a directed, and heterogeneous, network by redirecting the flows, thanks to the optimisation of well-designed loss functions, in particular the basic reproduction number of the model. We provide a final size relation linking the basic reproduction number to the epidemic final sizes, for diffusions around a reference diffusion with basic reproduction number less than $1$. Experimentally, we show control is possible for different topologies, network heterogeneity levels, and speeds of diffusion. Our experimental results highlight the relevance of the basic reproduction number loss, compared to more straightforward losses.
\end{abstract}

\section{Introduction}
\label{sec:introduction}

\input{./sources/introduction}

\section{Background and Contributions}
\label{sec:background}
\input{./sources/background}


\section{Final Size Relation}
\label{sec:fsize}
\input{./sources/fsize}

\subsection{Proof of the Final Size Relation, \reprop{fsize-r0-smaller-one}}
\label{sec:pfsize-relation}
\input{./sources/proof_fsize_relation}

\section{Flows Optimisation}
\label{sec:optim_flows}
\input{./sources/optim_flows}


\section{Numerical Simulations}
\label{sec:experiments}
\input{./sources/experiments}

\section{Conclusion, Future Works}
\label{sec:conclusion}
\input{./sources/conclusion}


\printbibliography

\newpage

\appendix

\section{Proofs of the Uniform Stability: \relem{unif-stab-dfe}}
\label{app:pstab-dfe}
\input{./appendices/proof_unif_stability.tex}



\section{SEPIR Model}
\label{app:sepir}

\input{./appendices/SEPIR_model.tex}

\end{document}

%% file: macros/environments_macros.tex
\newtheorem{theoreme}{Theorem}
\newtheorem{lemma}[theoreme]{Lemma}
\newtheorem{proposition}[theoreme]{Proposition}
\newtheorem{definition}[theoreme]{Definition}

\newtheorem{remarque}[theoreme]{Remarque}

\newcommand{\reeq}[1]{Equation~\eqref{eq:#1}}

\newcommand{\resec}[1]{Section~\ref{sec:#1}}
\newcommand{\redef}[1]{Definition~\ref{def:#1}}

\newcommand{\refig}[1]{Figure~\ref{fig:#1}}

\newcommand{\relem}[1]{Lemma~\ref{lem:#1}}

\newcommand{\reprop}[1]{Proposition~\ref{prop:#1}}

\newcommand{\reapp}[1]{Appendix~\ref{app:#1}}

\newcommand{\retab}[1]{Table~\ref{tab:#1}}

\newcommand{\resectwo}[2]{Sections~\ref{sec:#1} and~\ref{sec:#2}}

%% file: macros/text_macros.tex
\newcommand{\gu}[1]{``#1''}	

\renewcommand{\cite}[1]{\autocite{#1}}
\newcommand{\citet}[1]{\textcite{#1}}   

%% file: macros/maths_macros.tex
\newcommand{\enstq}[2]{\left\lbrace#1\left\lvert#2\right.\right\rbrace} 

\DeclareMathOperator*{\argmin}{arg\,min}    

\DeclareMathOperator{\id}{Id}   

\newcommand{\nrm}[1]{\left\lVert#1\right\rVert} 
\newcommand{\nrmop}[1]{\nrm{#1}_\mathrm{op}}    

\newcommand{\abs}[1]{\left\lvert#1\right\rvert} 

\newcommand\ps[2]{\left\langle #1, #2 \right\rangle}   

\newcommand{\diag}[1]{\mathrm{diag}\paren{#1}}  

\newcommand{\eps}{\varepsilon}

\newcommand{\bo}{\mathcal{B}}   

\newcommand{\paren}[1]{\left(#1\right)} 

\newcommand{\ba}{\begin{aligned}}   
\newcommand{\ea}{\end{aligned}}

\newcommand{\cpl}[2]{\paren{#1,\,#2}}   


%% file: macros/flow_redirection_macros.tex
\newcommand{\nods}{\mathcal{N}}	
\newcommand{\edgs}{\mathcal{E}}	

\newcommand{\node}{n}	

\newcommand{\pol}{\boldsymbol{\pi}}	
\newcommand{\polref}{{\pol_\mathrm{ref}}}	

\newcommand{\prm}{\theta} 
\newcommand{\eprm}{\Theta} 
\newcommand{\prmopt}{\prm^*}    

\newcommand{\ngm}{\mathbf{G}}    
\newcommand{\brn}{\mathcal{R}_0} 

\newcommand{\mdpl}{\mathbf{M}}   
\newcommand{\statiodfmat}[1]{\tilde{\mu}_{#1}}    
\newcommand{\statiod}{\statiodfmat{\mdpl}}    

\newcommand{\statiodref}{\tilde{\mu}_{\polref}}   

\newcommand{\statiodprm}{\tilde{\mu}_\prm}  

\newcommand{\orate}{f}  
\newcommand{\vecorate}{\boldsymbol{\orate}} 

\newcommand{\diagbeta}{\boldsymbol{\beta}} 
\newcommand{\diagdelta}{\boldsymbol{\delta}} 
\newcommand{\diaggamma}{\boldsymbol{\gamma}} 
\newcommand{\diagalpha}{\boldsymbol{\alpha}} 

\newcommand{\mdplref}{\mdpl_{\mathrm{ref}}}  
\newcommand{\statiodmref}{\statiodfmat{\mdplref}}   

\newcommand{\initconta}{\mathfrak{I}_0}   

\newcommand{\epidaccr}{\textsc{Epiloss}\xspace}
\newcommand{\alphalaccr}{\textsc{NoDiffloss}\xspace}
\newcommand{\rapsumaccr}{\textsc{QuickDiffLoss}\xspace}

\newcommand{\epidloss}{\epidaccr}    
\newcommand{\rapsumloss}{\rapsumaccr}
\newcommand{\alphaloss}{\alphalaccr}

\newcommand{\epidpolaccr}{\textsc{Epipol}\xspace}
\newcommand{\alphapolaccr}{\textsc{NoDiffpol}\xspace}
\newcommand{\rapsumpolaccr}{\textsc{QuickDiffpol}\xspace}

%% file: macros/experiment_macros.tex
\newcommand{\erdosrenyi}{Erd\H{o}s-Rényi\xspace}
\newcommand{\brbalbert}{Barab\'asi-Albert\xspace}
\newcommand{\waxman}{Waxman\xspace}
\newcommand{\relcav}{Relaxed Caveman}



%% file: macros/authors_citations_macros.tex
\newcommand{\citepz}{}

\newcommand{\citenyioha}{}

\newcommand{\citehan}{}

\newcommand{\citepj}{}

\newcommand{\citepzejpthirteen}{}
\newcommand{\citepzejpfourteen}{}
\newcommand{\citenpp}{}

\newcommand{\citekermack}{}

\newcommand{\citemagalsixteen}{}
\newcommand{\citemagaleighteen}{}

\newcommand{\citegao}{}

\newcommand{\citearinoseventeen}{}

\newcommand{\citebeaufort}{}
\newcommand{\citecaswell}{}

\newcommand{\citegolub}{}

\newcommand{\citealizon}{}

\newcommand{\citejunling}{}

\newcommand{\citediekmannnetal}{}

%% file: sources/introduction.tex
Networks are critical infrastructures, whether they are for instance transportation networks \cite{youn2008}, telecommunication networks \cite{pastorsatorras2004, newman2002}, or supply networks \cite{perfido2017}. Unfortunately, they may be invaded by undesirable processes, such as diseases \cite{vanmieghem2014inho} or virus malwares \cite{garetto2003}.
A common framework to model these processes is that of systems of coupled ordinary differential equations, where the equations running in each node (representing populations such as cities or countries) are coupled in some way by the network \cite{nowzari15}. The ODE's they are based on are deterministic, compartmental models, which were introduced at the beginning of the 20\textsuperscript{th} century, following notably \citet{kermack}, and described originally how individuals transitioned from state to state --- healthy, infected, recovered, for instance --- when confronted to a disease \cite{princeton}.

Two majors classes of coupling have emerged. On the one hand, the interactions between populations may be described by a static contact structure \cite{pastor2001, pastor2015, nowzari15}, where individuals in a node remain in the node, and may be infected by people in the neighbouring nodes. On the other hand, in the epidemic reaction-diffusion models, also known as metapopulation models with explicit movement \cite{arino09}, individuals can only be infected by other individuals in the same node, but are allowed to move to neighbouring nodes. Following their apparition in ecology \cite{levins1969}, these models have sparked considerable interest in mathematical epidemiology \cite{brauer2001, vddriessche2002, wang2005, allen07, arino09, tien2015, arino17, bichara2017, gao2019, gao2020}. 

A crucial issue is the control of these undesirable processes. As far as contact models are concerned, a first option to contain the spread is to act on the disease parameters, as the infection rate and the curing rate \cite{gourdin2011, preciado2013optvacc, preciado2014, nowzari2017}. For instance, treatments may reduce the likelihood to get infected, or speed up recovery.
Another option is to modify the network structure. Notably, \citepz \citet{preciado2013} reduce the flows between cities in order to limit the spread of an epidemic. Yet, perhaps surprisingly, the control problem has not been studied so far for reaction-diffusion models. In our work, we therefore propose to address this question. We act on the diffusion between the nodes, rather than on the disease parameters: indeed, diffusion is the specific property of these models, and therefore it makes sense focusing on it. Moreover, diffusion is a control variable truly accessible at all times, even when treatments are not available.

We aim at reducing the final size of the epidemic, that is the final number of individuals who have been infected. Rather than lower the flows, which would seem an obvious way to stop the spread, we merely redirect them: indeed, redirection represents a lighter alteration of the network structure, and can help it keep functioning as normally as possible, which is desirable.
We use the basic reproduction number of the system (introduced in \citediekmannnetal \citet{diekmann90}, see also \citediekmannnetal \citet{MAT09} for further explanations) as criterium to redirect the flows. We aim at minimizing it by gradient descent, with respect to some relevant parameterisation of the diffusion.

We start by giving some background material, discussing related works, and presenting our contributions, in \resec{background}. We then provide a final size relation linking the basic reproduction number, and the final size of the epidemic, in \resec{fsize}. Next, we explicit the optimisation problem in \resec{optim_flows}. Finally, we present the results of our numerical simulations in \resec{experiments}. The code for the simulations, written in Python, is available on the git repository \url{https://reine.cmla.ens-cachan.fr/masse/flow_redirection}.

%% file: sources/background.tex
We start by recalling the definition of metapopulation models with diffusion, and the basic reproduction number criterion for stability (\resec{metapop_models}). Next, we discuss related works on epidemic control (\resec{state_art}) and final size relations (\resec{fsr}). Finally, we state our contributions (\resec{contributions}).

\subsection{Metapopulation Models with Diffusion}
\label{sec:metapop_models}
Let $\mathcal{G}=\cpl{\nods}{\edgs}$ be a strongly connected, directed graph, with nodes set $\nods$, and edges set $\edgs$. 
For each node $\node \in \nods$, we write $\beta_\node$, $\delta_\node$ and $\gamma_\node$ the positive infection, incubation and curing rates respectively, of a scalar SEIR model \cite{princeton}. We write $\diagbeta = \diag{\beta_1,\ldots,\beta_{|\mathcal{N}|}}$ the corresponding diagonal matrix, and likewise for the other coefficients. 
Capital letters like $S$, $E$, $I$ or $R$ are vectors of size $|\mathcal{N}|$, such that for instance, $S_\node$ is the numbers of individuals in compartment \gu{S} of node $\node\in\nods$.
Coupling between nodes is realised by a diffusion matrix, which definition we now recall.
\begin{definition}[Diffusion Matrix]
\label{def:diff_mat}
A diffusion matrix $\mdpl$ on $\mathcal{G}$ first has nonzero off-diagonal entries only for coordinates $\cpl{i}{j}$ such that the edge $i\leadsto j$ belongs to $\edgs$. Secondly, it is Metzler, that is for $i,j\in \nods$, $i\neq j$, we have $\mdpl_{ij}\geq 0$. Then, it is irreducible. Finally, its columns sum to zero. 
\end{definition}
Since $\mathcal{G}$ is strongly connected, such matrices do exist.
Standard Perron-Frobenius theory guarantees that a diffusion matrix $\mdpl$ admits a stationary distribution $\statiod$, that is a positive right eigenvector such that $\mdpl \statiod = 0$, and which coordinates sum to $1$.
The reaction-diffusion extension of the standard SEIR system to a network evolves according to, for all $t\geq 0$,
\begin{equation}
\label{eq:seir-metapop}
\left\lbrace \ba
\frac{dS}{dt} &= - \diagbeta S \odot I + \mdpl S \\
\frac{dE}{dt} &= \diagbeta S \odot I - \diaggamma E + \mdpl E\\
\frac{dI}{dt} &= \diaggamma E  - \diagdelta I + \mdpl I \\
\frac{dR}{dt} &= \diagdelta I + \mdpl R
\ea \right.
\end{equation}
where, for two vectors $U$ and $V$ of equal dimensions, we write $U\odot V$ their coordinate wise product vector, that is $U\odot V = \paren{U_\node V_\node}_\node$. 
For instance, for a node $\node\in\nods$, the equation on $S_\node$ reads: $d S_\node/dt = -\beta_\node S_\node(t)I_\node(t) + \sum_{i=1}^{|\mathcal{N}|} \mdpl_{\node, i} S_i(t)$. 
Standard results --- see \citet{arino09} and references therein --- guarantee that, for all nonnegative initial condition $\big(S(0), E(0), I(0), R(0)\big)$, the solution to \reeq{seir-metapop} is global, remains nonnegative, and converges to a fixed point of the form $\big(S(\infty),0,0,R(\infty) \big)$. Moreover, the total population is preserved: $\sum_{\node\in\nods} S_\node(t)+E_\node(t)+I_\node(t)+R_\node(t)$ is constant. In what follows, we assume it equals $1$.

The \gu{Disease Free Equilibrium} (DFE) $\paren{\statiod, 0, 0, 0}$ is a fixed point of \reeq{seir-metapop}, where there is no disease: all individuals are in the compartment $S$. We want it to be stable, and therefore we recall here a well-studied stability criterion, which we make extensive use of. The stability of the DFE is governed by the spectral radius of the next-generation matrix \cite{diekmann90, MAT09}, which is called the basic reproduction number, and written $\brn$. The DFE is stable if, and only if, we have $\brn < 1$ \cite{arino05}. 
The next-generation matrix in the sense of \citet{MAT09} associated with the system of \reeq{seir-metapop} is $\ngm_\mdpl := \diagbeta \,\diag{\statiod}\,\paren{\mdpl-\diagdelta}^{-1}\,\diaggamma\,\paren{\mdpl-\diaggamma}^{-1}$, where $\diag{\statiod}$ is the diagonal matrix which diagonal coefficients are those of $\statiod$. The basic reproduction number depends on the diffusion matrix $\mdpl$, and we write it $\brn = \brn(\mdpl) = \rho\paren{\ngm_\mdpl}$, where $\rho$ designates the spectral radius.

Using the basic reproduction number as a stability criterion offers several advantages. First, taking it below $1$ is equivalent to ensuring the eigenvalues of the Jacobian of the system of \reeq{seir-metapop} at the DFE have negative real parts, which is the most straightforward stability criterion. Second, since it is the spectral radius of the next-generation matrix, it is obtained through the study of a matrix of order $|\mathcal{N}|$, whereas the Jacobian is of order $4|\mathcal{N}|$. Finally, since flows are not symmetrical, the Jacobian is not symmetrical either. However, the next-generation matrix is positive, so that its spectral radius is differentiable, and an analytical formula exists for its derivative (see \resec{diff_losses}).

Finally, we use the following notion of policy over the network.
\begin{definition}[Policy over a Network]
We call policy a stochastic matrix $\pol$ of order $|\nods|$ such that, for every node $\node\in\nods$, the row $\paren{\pol_{\node,i}}_{i\in\nods}$ is a probability distribution over $\nods$, and such that $\pol_{\node,i}$ is nonzero if, and only if, $\edgs$ contains an edge $\node \leadsto i$.
\end{definition}

\subsection{Network Deterministic Epidemic Control}
\label{sec:state_art}
Control of deterministic epidemic processes on networks is an alive direction of research: see for instance the review \citet{nowzari15}. A first option is to consider dynamic, or online, controls \cite{ruan2012, enyioha2013, han2015}. For instance, \citenyioha \citet{enyioha2013} propose a \emph{bio-inspired} strategy where nodes are allowed to to go into \emph{sleep} or \emph{dormant} states which reduce their susceptibility to the disease, and compute the optimal probabilities of nodes going into this state to prevent a small infection resulting in an epidemic in the network. \citehan \citet{han2015} adopt a different approach by defining a set of networks consistent with early observed data, then finding the optimal allocation of resources to control the worst-case spread that can take place in the aforementioned set of networks.

On the other hand, the control may be offline. One common strategy consists in reducing the maximum real part of the eigenvalues of the Jacobian of the system at the DFE, so as to ensure its exponential stability. The Jacobian typically writes $J=\diagbeta A - \diagdelta$, where $A$ is the adjacency matrix of the network, and $\diagbeta$, $\diagdelta$ contain epidemiological parameters linked with the disease. Two strategies are then available. On the one hand, one may act on the network structure, that is on $A$. For instance, \citepj \citet{preciado2009} devise well-suited values for the connectivity radius of a random geometric graph, ensuring stability of the DFE, while \citepz \citet{preciado2013} reduce the flows between cities. On the other hand, the control may be exercised on the epidemic parameters. In \citepzejpthirteen \citet{preciado2013optvacc} and \citepzejpfourteen \citet{preciado2014}, the infectivity and curing parameters, gathered in $\diagbeta$ and $\diagdelta$, are optimised under two objectives. First, for a given budget allocated to the tuning of the parameters, the authors look to minimise as much as possible the maximum real part of the eigenvalues. Conversely, they look for the minimal budget ensuring this maximum real part is taken below some threshold. In particular, \citenpp \citet{nowzari2017} show these optimisation problems may be cast as geometric programs, thus allowing their solving with standard solvers.

\subsection{Final Size Relations}
\label{sec:fsr}

The final size of an epidemic is the asymptotic number of individuals confronted to the disease. For the model of \reeq{seir-metapop}, it is given by: $\sum_{\node\in\nods} R_\node(\infty)$, that is the sum over all nodes $\node\in\nods$ of the final number of recovered individuals. The final size is a crucial outcome of the epidemic: estimating it and studying its dependency on model parameters, has been given plenty of attention: see notably the review \citejunling \citet{ma2006}. \citekermack \citet{kermack} established the following well-known equation for a scalar, deterministic model: the final size $r(\infty)$ (the lowercase emphasises it is a real number) and the basic reproduction number $\brn$ are linked by 
\begin{equation}
\label{eq:fsize_sir_scalar}
\brn r(\infty) + \log\paren{1-r(\infty)} = 0.
\end{equation}
One crucial property of \reeq{fsize_sir_scalar} is the monotonous link it shows between $\brn$, which concerns the onset of the epidemic, and its final size, which concerns its outcome. Several studies have since been devoted to extending this relation to more evolved models \cite{arino2007}. \citemagalsixteen \citet{magal2016} and \citemagaleighteen \citet{magal2018} study the final size of a multi-group SIR epidemic model. Other works study it for mixing models on networks \cite{andreasen2011, brauer2008}. 
Finally, on a slightly different perspective, \citegao \citet{gao2020dispersal} studies the sensitivity of the size of the endemic equilibrium with respect to variables of interest, such as the rate of diffusion, or the basic reproduction number, in the specific context of a SIS model with diffusion. However, up to our best knowledge, there is no general result expressing the final size of an epidemic reaction-diffusion model as a function of its basic reproduction number.

\subsection{Contributions}
\label{sec:contributions}
We show we can control the epidemic system of \reeq{seir-metapop}, running on a directed, heterogeneous network, by redirecting flows of individuals.
Theoretically, we provide a final size relation linking the basic reproduction number and the final size of the epidemic (\reprop{fsize-r0-smaller-one}). It applies to reaction-diffusion processes with diffusion matrices close to a reference diffusion matrix, whose reproduction number is strictly less than $1$. To obtain this result, we prove a uniform stability result (\relem{unif-stab-dfe}) which extends uniformly, in some neighbourhood of the reference diffusion matrix, the standard stability criterion given by the next-generation method.

Then, we present our methodology based on the control of the basic reproduction number to control the epidemic spread. We design a parameterisation of the diffusion allowing us to redirect the flows, and define several losses we compare to the basic reproduction number loss. The losses are differentiable so that, even if the optimisation problem is non linear and not symmetrical, we can solve it by gradient descent.

Finally, we validate our approach on numerical simulations, with synthetic data presenting different topologies, different levels of network heterogeneity, and a range of diffusion speeds. The procedure works for general reaction models: for the sake of simplicity, we conduct the analysis with SEIR, but also carry on experiments with a more complex SEPIR reaction. 

%% file: sources/fsize.tex
We first present our final size relation, together with a crucial intermediary lemma in \resec{brn_fsize}. We then prove the relation in \resec{pfsize-relation}.

\subsection{Linking Basic Reproduction Number and Final Size}
\label{sec:brn_fsize}
Consider some diffusion matrix $\mdplref$ such that $\brn(\mdplref)<1$. Then, for diffusion matrices $\mdpl$ close enough to $\mdplref$, the ratio between the final size and the initial number of infected individuals is controlled by monotonic functions of the basic reproduction numbers $\brn(\mdpl)$. Let us first define two quantities needed to state formally this result. For every diffusion matrix $\mdpl$, we write $v^\mdpl$ a non-negative eigenvector (thus, not zero) summing up to $1$, associated to the spectral radius of the next-generation matrix with large domain \cite{MAT09}. Vectors $v^\mdpl$ exist because this matrix is non-negative \cite{meyer2000}.
Then, for every couple $\cpl{E^\mdpl(0)}{I^\mdpl(0)}$ of vectors, we write
$\mathfrak{I}^\mdpl (0) = \sum_{\node\in\nods} \big(E^\mdpl_{\node}(0) + I^\mdpl_{\node}(0)\big)$ the initial number of individuals either exposed or infected. We can now state our final size relation.
\begin{proposition}[Final size relation]
\label{prop:fsize-r0-smaller-one}
Let $\mdplref$ be a diffusion matrix, such that $\brn(\mdplref) < 1$. Then, for every $\eps>0$ small enough, there exists a ball $\bo$ of diffusion matrices around $\mdplref$, and $\eta >0$ such that, for every $\mdpl\in\bo$, for every initial condition 
$\nrm{\big(S^\mdpl (0), E^\mdpl(0), I^\mdpl(0), R^\mdpl(0)\big) - \big(\statiodmref, 0, 0, 0\big)} < \eta$,
such that the relation $\cpl{E^\mdpl(0)}{I^\mdpl(0)}= \initconta^\mdpl(0) v^\mdpl$ holds for some vector $v^\mdpl$, we have 
\begin{multline}
\label{eq:fsize}
\frac{\initconta^\mdpl(0)}{1 - \paren{1-\varepsilon}\brn(\mdpl)} \leq \sum_{\node\in\nods}R_\node^\mdpl(\infty) \\
\leq \frac{\initconta^\mdpl(0)}{1 - \paren{1+\varepsilon}\brn(\mdpl)},
\end{multline}
where $R^\mdpl(\infty)$ is the asymptotic vector of individuals in compartment $R$ of the solution of \reeq{seir-metapop}.
\end{proposition}
This result shows that, for diffusions $\mdpl$ close enough to $\mdplref$, the final size is controlled by monotonic functions of the basic reproduction numbers $\brn(\mdpl)$'s: the smaller the $\brn(\mdpl)$'s, the closer the final size gets to the initial number of infected people, $\initconta^\mdpl(0)$, meaning the epidemic has not spread widely within the population.
Note that the proof could extend to other compartmental models, like SIR or SEPIR (see the supplementary materials for a description of the latter).

Our relation has two main limits. First, it only applies for $\brn(\mdplref)<1$, that is when the DFE is stable, but even in this case, the final size may be deemed unacceptably high, so that studying if it can be lowered remains relevant.
Then, we restrict to initial conditions of the form $\mathfrak{I}^\mdpl (0) v^\mdpl$. However, by doing so, we only constrain their orientation, but not the absolute number of individuals $\initconta^\mdpl (0)$, which can be any sufficiently small value.

We have therefore addressed (albeit partially) the question of the extent to which the basic reproduction number influences the final size. Another question is the extent to which redirecting the flows modifies the basic reproduction number. We do not address theoretically it here, but the subsequent sections show experimentally that by redirecting the flows, we manage to reduce the basic reproduction number, and the corresponding final size. See for instance \citegao \citet{gao2020dispersal} for a work studying the influence of the diffusion on the basic reproduction number. 
Before proceeding further, we first prove \reprop{fsize-r0-smaller-one}.

%% file: sources/proof_fsize_relation.tex
In the proof, we use \relem{unif-stab-dfe}, which shows the DFE is uniformly stable at the vicinity of $\mdplref$, and which is stated and proved in \reapp{pstab-dfe}.
\paragraph{Comparison with \gu{constant-matrices}}
Let $\eps >0$. Thanks to \relem{unif-stab-dfe}, we can find a ball $\bo$ around $\mdplref$, and $\eta > 0$ such that, for every initial condition $\big(S(0), E(0), I(0), R(0)\big)$ satisfying 
$\nrm{\big(S(0), E(0), I(0), R(0)\big) - \big(\statiodmref, 0, 0, 0\big)} < \eta$,
for any diffusion matrix $\mdpl$ in $\bo$, for all $t\geq 0$, we have
$\nrm{\paren{S^\mdpl(t), E^\mdpl(t), I^\mdpl(t), R^\mdpl(t)} - \paren{\statiodfmat{\mdpl}, 0, 0, 0}} < \eps$.
Let $\mdpl\in\bo$. As a consequence, for all $t\geq 0$, we have, coordinate-wise,
\begin{align*}
&\begin{pmatrix}
\mdpl-\diaggamma & \diagbeta \mathrm{diag}\paren{\Tilde{\mu}_\mdpl \paren{1-\eps}} \\
\diaggamma & \mdpl-\diagdelta
\end{pmatrix} \\
\leq
&\begin{pmatrix}
\mdpl-\diaggamma & \diagbeta \mathrm{diag}\paren{S^\mdpl(t)} \\
\diaggamma & \mdpl-\diagdelta
\end{pmatrix} \\
\leq
&\begin{pmatrix}
\mdpl-\diaggamma & \diagbeta \mathrm{diag}\paren{\Tilde{\mu}_\mdpl \paren{1+\eps}} \\
\diaggamma & \mdpl-\diagdelta
\end{pmatrix}.
\end{align*}
Let us define
\begin{equation*}
\ba
\mathbf{F}_\mdpl=
\begin{pmatrix}
0 & \diagbeta \, \mathrm{diag}\paren{\Tilde{\mu}_\mdpl} \\
0 & 0
\end{pmatrix} 
\text{, and} \quad 
\mathbf{V}_\mdpl = \begin{pmatrix}
\mdpl -\diaggamma & 0 \\
\diaggamma & \mdpl - \diagdelta
\end{pmatrix}.
\ea
\end{equation*}
To proceed by comparison, we consider the following two dynamics:
\begin{equation*}
\left\lbrace \ba
\cpl{\frac{d E_-^\mdpl}{dt}}{\frac{dI_-^\mdpl}{dt}} &=
\paren{\paren{1-\eps}\mathbf{F}_\mdpl + \mathbf{V}_\mdpl}
\cpl{E_-^\mdpl(t)}{I_-^\mdpl(t)}\\
\cpl{\frac{d E_+^\mdpl}{dt}}{\frac{dI_+^\mdpl}{dt}} &=
\paren{\paren{1+\eps}\mathbf{F}_\mdpl + \mathbf{\mathbf{V}}_\mdpl}
\cpl{E_+^\mdpl(t)}{I_+^\mdpl(t)},
\ea \right.
\end{equation*}
both starting at $\cpl{E^\mdpl(0)}{I^\mdpl(0)}$. In the following, we consider fixed a vector $v^\mdpl$ such that $\cpl{E^\mdpl(0)}{I^\mdpl(0)}$ is directed by $v^\mdpl$.
Now, $\paren{1-\eps}\mathbf{F}_\mdpl+\mathbf{V}_\mdpl$ and $\paren{1+\eps}\mathbf{F}_\mdpl+\mathbf{V}_\mdpl$ are both Metzler matrices, since $\mathbf{V}_\mdpl$ is non-negative off diagonal and $\mathbf{F}_\mdpl$ has non-negative coefficients. As a result, we may use the comparison Theorem B.1 of \citet{smith95}, and we obtain that, for all $t\geq 0$,
$\cpl{E_-^\mdpl(t)}{I_-^\mdpl(t)}
\leq
\cpl{E^\mdpl(t)}{I^\mdpl(t)}
\leq
\cpl{E_+^\mdpl(t)}{I_+^\mdpl(t)}$.
Now, the matrices in the bounding systems are constant, so we can express their solutions with the matrix exponential and, as a result, for all $t\geq 0$, we have
\begin{align*}
\exp \big(((1-\eps) & \mathbf{F}_\mdpl + \mathbf{V}_\mdpl) t\big) \cpl{E^\mdpl(0)}{I^\mdpl(0)}\\
\leq
&\cpl{E^\mdpl(t)}{I^\mdpl(t)} \\
\leq
&\exp \big(((1+\eps)\mathbf{F}_\mdpl + \mathbf{V}_\mdpl) t\big)
\cpl{E^\mdpl(0)}{I^\mdpl(0)}.
\end{align*}
\paragraph{Integrating the comparison}
Now, for every matrix $A$ which eigenvalues all have a negative real part, we know that
$\int_0^{\infty} \exp \paren{As}ds = -A^{-1}$.
Since $\brn(\mdplref)<1$, all the eigenvalues of $F_{\mdplref}+V_{\mdplref}$ have real parts (strictly) less than $1$. Now, the eigenvalues of a matrix depend continuously on the matrix. Upon diminishing the ball $\bo$ around $\mdplref$, we may therefore assume that, for all diffusion matrix $\mdpl$ in $\bo$, all the eigenvalues of $\mathbf{F}_\mdpl+\mathbf{V}_\mdpl$ have real parts - strictly - less than $1$. Upon diminishing $\eps$, we may finally assume that, for all diffusion matrix $\mdpl\in\bo$, the eigenvalues of $\paren{1-\eps}\mathbf{F}_\mdpl+\mathbf{V}_\mdpl$ and $\paren{1+\eps}\mathbf{F}_\mdpl+\mathbf{V}_\mdpl$ have - strictly - negative real parts. 
Thus, writing $\mathbf{I}=\int_0^{\infty}
\cpl{E^\mdpl(s)}{I^\mdpl(s)}\, ds$, we have
\begin{multline*}
 -\paren{\paren{1-\eps}\mathbf{F}_\mdpl + \mathbf{V}_\mdpl}^{-1}
\cpl{E^\mdpl(0)}{I^\mdpl(0)}
\leq \mathbf{I} \\
\leq
 -\paren{\paren{1+\eps}\mathbf{F}_\mdpl + \mathbf{V}_\mdpl}^{-1}
\cpl{E^\mdpl(0)}{I^\mdpl(0)}.
\end{multline*}
\paragraph{Using the next-generation matrix}
Let us now multiply this inequality by a square non-negative matrix $\mathbf{Z}$, to be fixed below. We therefore have 
\begin{multline}
\label{eq:ineq_with_Z}
 -\mathbf{Z}\paren{\paren{1-\eps}\mathbf{F}_\mdpl + \mathbf{V}_\mdpl}^{-1}
\cpl{E^\mdpl(0)}{I^\mdpl(0)}
\leq  \mathbf{Z} \mathbf{I} \\
\leq
 -\mathbf{Z}\paren{\paren{1+\eps}\mathbf{F}_\mdpl + \mathbf{V}_\mdpl}^{-1}
\cpl{E^\mdpl(0)}{I^\mdpl(0)}.
\end{multline}
Using the invertibility of $\mathbf{V}_\mdpl$, we know that
$- \paren{\paren{1-\eps} \mathbf{F}_\mdpl + \mathbf{V}_\mdpl}
=\paren{-\paren{1-\eps}\mathbf{F}_\mdpl \mathbf{V}_\mdpl^{-1} - \id}\mathbf{V}_\mdpl$.
Since the next-generation matrix with large domain \cite{MAT09}, $\mathbf{K}_\mdpl = -\mathbf{F}_\mdpl \mathbf{V}_\mdpl^{-1}$ is non-negative, we can find a right eigenvector $v^\mdpl\neq 0$ associated to the spectral radius $\brn=\brn(\mdpl)=\rho\paren{\mathbf{K}_\mdpl} = \rho\paren{\mathbf{G}_\mdpl}$. In what follows, we fix one such $v = v^\mdpl$, and drop the explicit dependencies on $\mdpl$ for $v^\mdpl$ and $\brn(\mdpl)$ so as to simplify the notations. Hence, we obtain
$- \paren{\paren{1-\eps} \mathbf{F}_\mdpl + \mathbf{V}_\mdpl}\mathbf{V}_\mdpl^{-1}v
=\paren{1-\eps}\brn v - v$.

Since $\brn(\mdplref)<1$, and since the spectral radius of a matrix depends continuously on the matrix, upon reducing $\bo$ further, we may assume that, for all diffusion matrix $\tilde{\mdpl}\in\bo$, we have
$\paren{1+\eps}\brn(\tilde{\mdpl}) < 1$,     
so that both $1-\paren{1+\eps}\brn(\tilde{\mdpl})$ and, \emph{a fortiori}, $1-\paren{1-\eps}\brn(\tilde{\mdpl})$ do not vanish. 

By assumption, $\cpl{E^\mdpl(0)}{I^\mdpl(0)}$ is directed by $v$, 
so that it is also a right eigenvector of $\mathbf{K}_\mdpl$ associated to $\brn(\mdpl)$.
Hence, we successively have 
\begin{multline*}
- \paren{\paren{1-\eps} \mathbf{F}_\mdpl + \mathbf{V}_\mdpl}\mathbf{V}_\mdpl^{-1} \cpl{E^\mdpl(0) }{I^\mdpl(0)}\\
=\paren{\paren{1-\eps}\brn - 1}
\cpl{E^\mdpl(0)}{I^\mdpl(0)}
\end{multline*}
then, multiplying by $\paren{\paren{1-\eps}\mathbf{F}_\mdpl+\mathbf{V}_\mdpl}^{-1}$, we obtain
\begin{multline*}
- \mathbf{V}_\mdpl^{-1}
\cpl{E^\mdpl(0)}{I^\mdpl(0)}\\
=\paren{\paren{1-\eps}\brn - 1}
\paren{\paren{1-\eps} \mathbf{F}_\mdpl + \mathbf{V}_\mdpl}^{-1}
\cpl{E^\mdpl(0)}{I^\mdpl(0)}.
\end{multline*}
Finally, dividing by $1-\paren{1-\eps}\brn$ - which is not zero thanks to the above - and multiplying by the matrix $\mathbf{Z}$, we get
\begin{multline*}
\paren{\paren{1-\eps}\brn-1}^{-1} \mathbf{Z} \mathbf{V}_\mdpl^{-1} \cpl{E^\mdpl(0)}{I^\mdpl(0)}\\
= -\mathbf{Z} \paren{\paren{1-\eps} \mathbf{F}_\mdpl + \mathbf{V}_\mdpl}^{-1}
\cpl{E^\mdpl(0)}{I^\mdpl(0)}.
\end{multline*}
The same reasoning applies with the matrix $\paren{1+\eps}\mathbf{F}_\mdpl + \mathbf{V}_\mdpl$ in the upper-bounding system.
Plugging it all back into \reeq{ineq_with_Z}, we obtain
\begin{multline*}
-\paren{1-\paren{1-\eps}\brn}^{-1} \mathbf{Z} \mathbf{V}_\mdpl^{-1}
\cpl{E^\mdpl(0)}{I^\mdpl(0)}
\leq \mathbf{Z} \mathbf{I} \\
\leq
-\paren{1-\paren{1+\eps}\brn}^{-1} \mathbf{Z} \mathbf{V}_\mdpl^{-1}
\cpl{E^\mdpl(0)}{I^\mdpl(0)}.
\end{multline*}
Now, set 
$\mathbf{Z} = 
\begin{pmatrix}
0 & 0 \\
0 & \diagdelta
\end{pmatrix}$
and, denoting $\mathbf{\Lambda}_{(i:j)}$ the submatrix of $\mathbf{\Lambda}$ obtained by extracting the $i^{\text{\tiny th}}$ to $j^{\text{\tiny th}}$ rows of $\mathbf{\Lambda}$,
set
$\cpl{\mathbf{B}_1}{\mathbf{B}_2}=\paren{\mathbf{V}_\mdpl^{-1}}_{\paren{\abs{\nods}+1:2\abs{\nods}}}$,
a $|\nods| \times 2|\nods|$ matrix.
Restraining to the rows ${\paren{\abs{\nods}+1:2\abs{\nods}}}$ corresponds to the infected coordinates - that of the vector $I^\mdpl$. Thus, we obtain
\begin{multline}
\label{eq:bound_integralI}
\frac{- \cpl{\diagdelta \mathbf{B}_1}{\diagdelta \mathbf{B}_2}}{1-\paren{1-\varepsilon}\brn} 
\cpl{E^\mdpl(0)}{I^\mdpl(0)}
\leq \diagdelta \int_0^{\infty}I^\mdpl(s) ds \\
\leq 
\frac{- \cpl{\diagdelta \mathbf{B}_1}{\diagdelta \mathbf{B}_2}}{1-\paren{1+\eps}\brn}  
\cpl{E^\mdpl(0)}{I^\mdpl(0)}.
\end{multline}
\paragraph{Conclusion}
Remember we want to bound
$\Delta^\mdpl = \sum_{\node\in\nods}R_\node(\infty) = 
\sum_{\node\in\nods} \delta_\node \int_0^\infty I^\mdpl_\node(s) ds
= \int_0^\infty \boldsymbol{e}^T \diagdelta I^\mdpl(s) ds$,
where $\boldsymbol{e} = \paren{1, \dots ,1}$ is the vector with all entries equal to unity. 
Multiplying by $\boldsymbol{e}$ every term of \reeq{bound_integralI}, using an explicit expression of $\mathbf{V}_\mdpl^{-1}$ - inversion of a $2\times2$ block matrix - and expanding the terms $\boldsymbol{e}^T \cpl{\diagdelta \mathbf{B}_1}{\diagdelta \mathbf{B}_2} \cpl{E^\mdpl(0)}{I^\mdpl(0)}$, we obtain
\begin{multline*}
\paren{1 - \paren{1+\eps}\brn} \Delta^\mdpl \leq \\
\boldsymbol{e}^T\paren{\diagdelta \paren{\mdpl - \diagdelta}^{-1} \diaggamma \paren{\mdpl - \diaggamma}^{-1} E^\mdpl(0) - \diagdelta \paren{\mdpl - \diagdelta}^{-1} I^\mdpl(0)} \\
\leq \paren{1 - \paren{1-\eps}\brn} \Delta^\mdpl.
\end{multline*}
Using the fact $\diagdelta \paren{\mdpl-\diagdelta}^{-1}=-\id + \mdpl\paren{\mdpl-\diagdelta}^{-1}$,
and the analogous relation for $\diaggamma\paren{\mdpl-\diaggamma}^{-1}$,
we may rewrite the term between parentheses as
\begin{multline*}
 \paren{\mdpl \paren{\mdpl - \diagdelta}^{-1} - \id } \paren{\mdpl \paren{\mdpl - \diaggamma}^{-1} - \id } E^\mdpl(0) \\
 + \paren{\id - \mdpl\paren{\mdpl-\diagdelta}^{-1}} I^\mdpl(0).
\end{multline*}
Since $\mdpl$ is a diffusion matrix, we know that the coordinates of the vectors $\mdpl \paren{\mdpl - \diagdelta}^{-1} E^\mdpl(0)$, $\mdpl \paren{\mdpl - \diaggamma}^{-1} E^\mdpl(0)$, $\mdpl \paren{\mdpl - \diagdelta}^{-1}\mdpl \paren{\mdpl - \diaggamma}^{-1} E^\mdpl(0)$ and $\mdpl \paren{\mdpl - \diagdelta}^{-1} I^\mdpl(0)$ sum to zero. Hence, we finally obtain 
\begin{multline*}
\paren{1 - \paren{1-\eps}\brn}^{-1}{\boldsymbol{e}^T}\paren{E^\mdpl(0) + I^\mdpl(0)} \leq \Delta^\mdpl \\
\leq \paren{1 - \paren{1+\eps}\brn}^{-1} {\boldsymbol{e}^T} \paren{E^\mdpl(0) + I^\mdpl(0)}.
\end{multline*}
Expressing the scalar products with the vector $\boldsymbol{e}$ as a sum, and remembering $\brn=\brn(\mdpl)$, we obtain the expression given in the statement of the lemma.

%% file: sources/optim_flows.tex
We first define how we parameterise the diffusion (\resec{param_diff}), before defining the losses we optimise (\resec{losses_control_pol}). Then, we justify the losses are differentiable (\resec{diff_losses}). FInally, we discuss the algorithmic complexity of their optimisation (\resec{algo_complex}).

\subsection{Parameterising the Diffusion}
\label{sec:param_diff}
We aim at redirecting the flows in order to control the epidemic. To this avail, we let fixed the outrates in each node, that is the rate at which the individuals leave the nodes, and only modify the way they are dispatched along the several edges leaving the nodes, according to a policy. Therefore, we consider diffusion matrices of the form
\begin{equation*}
\mdpl(\prm) = \vecorate \, \big(\pol(\prm) - \id_{|\nods|}\big)^T,
\end{equation*}
where the outrate diagonal matrix $\vecorate$, and the policy $\pol(\prm)$, both in $\mathcal{M}_{\abs{\nods}}(\mathbb{R})$, are defined below.

For every node $\node \in \nods$, we call outrate in $\node$ a positive real number $\orate_\node$, and write $\vecorate = \mathrm{diag}(f_1,\ldots, f_{|\nods|})$. Let us define the parameter space
\begin{equation*}
\eprm = \mathbb{R}^{\abs{\edgs}}.
\end{equation*}
We index $\prm\in\eprm$ by the edges in $\edgs$: for instance, if $\node, i \in\nods$ are such that $\node\leadsto i$ belongs to $\edgs$, we write $\prm_{\node, i}$ the corresponding entry of $\prm$.
We use as control variable a policy $\pol(\prm)$ over $\nods$, parameterised with a parameter $\prm\in\eprm$, 
such that, for every node $\node$, $i\mapsto \pol_{\node,i}(\prm)$ is a softmax function over the neighbours of $\node$. Namely, for every edge $\node\leadsto i \in \edgs$, we impose
$
\pol_{\node, i}(\prm) = {e^{\prm_{\node, i}}}/{\sum_{j, \node\leadsto j \in \edgs} e^{\prm_{\node, j}}}
$, while for every node $i$ such that there is no edge $\node\leadsto i$, we impose $\pol_{\node,i}(\prm)=0$. 
As a result, for every $\prm$, $\pol(\prm)$ is indeed a policy, and the mapping $\prm \mapsto \pol(\prm)$ is regular. For every $\prm \in \eprm$, for every $\node\leadsto i\in\edgs$, the quantity $\pol_{\node, i}(\prm)$ is the proportion of individuals who leave the node $\node$ through the edge $\node\leadsto i$.

Let us finally check the diffusion matrices $\mdpl(\prm)$ are indeed diffusion matrices in the sense of \redef{diff_mat}. Let us fix $\prm$.
For every node $\node$, $i\mapsto \pol_{\node,i}(\prm)$ is a softmax function over the neighbours of $\node$ therefore, for every edge $\node\leadsto i\in\edgs$, $\pol_{\node,i}(\prm)$ is positive. Now, since $\mathcal{G}$ is strongly connected, we know $\pol(\prm)$ is irreducible, therefore $\mdpl(\prm)$ is as well, as $\vecorate$ is positive. Finally, the columns of $\mdpl(\prm)$ sum to $0$, as $\pol(\prm)$ is stochastic. This concludes our argument.

In what follows, we fix $\vecorate$.
Therefore, all quantities related to the diffusion are functions of $\prm$, and we write accordingly $\ngm(\prm)$, $\brn(\prm)$ and $\statiodprm$ the associated next-generation matrix, basic reproduction number, and stationary distribution, respectively.

\subsection{Losses, and Related Control Policies}
\label{sec:losses_control_pol}
We want to optimise the flows to control the epidemic: to this avail, we now introduce three losses on $\prm$. For each loss, the optimal parameter $\prmopt$ we obtain defines a control policy $\pol(\prmopt)$. The performances of these policies for epidemic control are evaluated in \resec{experiments}.

\textbf{Epidemic loss.} The main loss is the epidemic loss, defined by
\begin{equation*}
\epidaccr(\prm) = \brn(\prm).
\end{equation*}
The associated policy, $\pol(\prmopt)$, with $\prmopt \in \argmin_\prm \epidloss(\prm)$, is called the epidemic policy. It aims at stabilising the DFE, taking $\brn(\prm)$ below $1$, and reducing the final size: see \resec{brn_fsize}.

\textbf{No diffusion loss.} The second loss, or \alphalaccr for \gu{No Diffusion Loss}, is defined by, for every $\prm$,
\begin{equation*}
\alphaloss (\prm) = \mathcal{S}_{a} \biggl( \diagbeta\diagdelta^{-1}\statiodprm \biggl),    
\end{equation*}
where $\mathcal{S}_{a}: \mathbb{R}^{|\mathcal{N}|} \to \mathbb{R}$, is the $a$-smooth max function defined, for $a > 0$ and $u = (u_1,\ldots , u_{|\mathcal{N}|})$, by
\begin{equation*}
\mathcal{S}_{a}(u) = \frac{\sum_{i=1}^{|\mathcal{N}|} u_i\exp (a u_i)}{\sum_{i=1}^{|\mathcal{N}|}\exp (a u_i)}. 
\end{equation*}
It is a smoothed maximum of the individual basic reproduction numbers $\beta_\node \delta_\node^{-1} \statiodprm(\node)$ of each node $\node$, computed when there is no diffusion. The maximum of these reproduction numbers is therefore the limit, when $\tau \to \infty$, of the basic reproduction number of the system of \reeq{seir-metapop}, when the diffusion is replaced by $\mdpl/\tau$. Here, $\tau$ acts as the typical time at which diffusion occurs.

\textbf{Quick diffusion loss.} Conversely, the third loss, \rapsumaccr for \gu{Quick Diffusion Loss}, is the limit, when $\tau\to 0$, of the basic reproduction number (see \citebeaufort \citet{beaufort2021network} for a proof in the special case of a SIR reaction, but the same reasoning applies to SEIR). It is defined by, for every parameter $\prm$,  
\begin{equation*}
\rapsumloss (\prm) = \frac{\sum_{\node\in\nods} \beta_\node \statiodprm(\node)^2}{\sum_{\node\in\nods} \delta_\node \statiodprm(\node)}.
\end{equation*}

All three losses aim at first sight at redistributing the population on the network, by sending it to nodes with low $\beta_\node$ coefficient, and high $\delta_\node$ coefficient. This is the somewhat obvious strategy. Indeed, when $\beta_\node$ is small, and $\delta_\node$ is high, the epidemic is less severe. However, such a strategy, though enticing, is flawed. Indeed, increasing the number of individuals in a node increases its individual reproduction number $\beta_\node\delta_\node^{-1} \statiodprm(\node)$, so that putting all the population in the most favourable node will often not provide a good solution. Therefore, our three losses try to balance this objective with the negative effects increasing the population of nodes have. Moreover, through its direct dependency on $\mdpl(\prm)$, \epidaccr also takes into account the transfers of population which happen between the nodes during the epidemic. 

The policies obtained by minimising the three losses are called respectively \epidpolaccr, \alphapolaccr and \rapsumpolaccr.

\subsection{Differentiability of the Losses}
\label{sec:diff_losses}

\input{./sources/diff_losses.tex}

\subsection{Algorithmic Aspects}
\label{sec:algo_complex}
The parameter space $\eprm$ is of size at most $\abs{\nods}^2$ (in the case of a complete graph).
Given the softmax parameterisation introduced above, the differential of the stationary distribution is of size $|\mathcal{N}|^3$, while that of the next-generation matrix, used to compute the differential of $\brn(\prm)$, is of size $|\mathcal{N}|^4$. As a result, optimising \epidaccr is costlier than optimising the other two losses.
An explicit formula exists for the differential of $\statiodprm$ (see \citegolub \citet{golubmey86}) but in our experiments, we computed the differential through automatic differentiation, using the Python library TensorFlow. 

Losses are optimised using standard gradient descent, iterated for $400$ steps. We mainly used handcrafted stepsizes of the form $s(i) = \phi \sqrt{|\mathcal{N}|}\exp(\frac{\log(2)i}{\rho})$, where $0<i\leq 400$ is the iterate; typically $\rho = 250$ and $\phi = 2.5\times 10^{-3}$.

%% file: sources/diff_losses.tex
The optimisation of the losses we use gives nonlinear, and not symmetrical, optimisation problems, therefore we solve them by direct gradient descent (see \resec{exp_setup}). 
This is possible because the losses we use are differentiable with respect to $\prm$.
\begin{proposition}[Differentiability of the Losses]
\label{prop:diff_losses}
\epidaccr, \alphalaccr and \rapsumaccr are differentiable with respect to $\prm$ on $\eprm$.
\end{proposition}
Since all three losses depend on a smooth way on the basic reproduction number, and the stationary distribution, we only need to prove these are differentiable, which we do in \relem{diff_statiod_brn}. The fact the basic reproduction number is differentiable stems from its expression as the spectral radius of the next-generation matrix: it is also the spectral radius of the next-generation matrix with large domain \cite{MAT09}, but we cannot deduce its differentiability from it, as this latter matrix is not irreducible.
\begin{lemma}[Differentiability of Relevant Quantities]
\label{lem:diff_statiod_brn}
The map $\prm\mapsto \statiodprm$ is differentiable.
Moreover, the map $\prm\mapsto\brn\paren{\prm}$ is differentiable and, for every $\prm\in\eprm$, its euclidean gradient is given by 
\begin{equation*}
\frac{\partial \brn}{\partial \prm}\paren{\prm} = \paren{\ps{r(\prm)\,l(\prm)}{\frac{\partial }{\partial \prm_{i,j}}\ngm(\prm)}}_{1\leq i,j \leq |\mathcal{N}|},
\end{equation*}
where $\ps{\cdot}{\cdot}$ stands for the dot product on the space of matrices, and 
$l(\prm)$ and $r(\prm)$ are, respectively, a left and a right Perron-eigenvectors, (associated to the eigenvalue $\brn(\prm)$) of the next-generation matrix $\ngm(\prm)$, such that $l(\prm)r(\prm)=1$.
\end{lemma}

\begin{proof}
First, the stationary distribution is differentiable, as for $\delta t$ small enough, it is that of the irreducible stochastic matrix $\id_N + \delta t \, \mdpl\paren{\prm}^T$, which depends in a smooth way on the matrix \cite{meyer2000}, and $\prm\mapsto\mdpl(\prm)$ is smooth. 
Second, $\prm\mapsto\ngm\paren{\prm}$ is smooth. Indeed, from \resec{metapop_models} we know that 
$\ngm\paren{\prm} =
\diagbeta \, \mathrm{diag}\paren{\statiodprm}\,\paren{\mdpl(\prm) -\diagdelta}^{-1}\,\diaggamma\,\paren{\mdpl(\prm) -\diaggamma}^{-1}$.
Now, the maps $\prm\mapsto \statiodprm$ and $\prm\mapsto\mdpl\paren{\prm}$ are smooth, and so is matrix inversion. Third, for every $\prm\in\eprm$, $\ngm\paren{\prm}$ is positive. We said in \resec{metapop_models} that $\diagbeta$ and $\diaggamma$ are diagonal positive matrices. Moreover, we also recalled that, since $\mdpl\paren{\prm}$ is a diffusion matrix, $\statiodprm$ is a positive vector.
Finally, we know that $-\paren{\mdpl(\prm) -\diagdelta}^{-1}$ and $-\paren{\mdpl(\prm) -\diaggamma}^{-1}$ are positive: see for instance Lemma 1 of \citearinoseventeen \citet{arino17}. 
Fourth, $\prm\mapsto\brn(\prm)$ is differentiable. Indeed, thanks to \citecaswell \citet{caswell19}, the map $\ngm\mapsto\rho\paren{\ngm}$ is differentiable on the set of non-negative, irreducible matrices. Since $\ngm(\prm)$ is positive, it is \emph{a fortiori} irreducible. The formula then follows from \citecaswell \citet{caswell19} and the chain rule.
\end{proof}

%% file: sources/experiments.tex
First, we describe the experimental set-up in \resec{exp_setup}. Then, we study numerically the relation between the final size, and the basic reproduction number, in \resec{influence_brn}. We compare the overall performances of the policies for various graphs sizes and topologies in \ref{sec:direct_comparison}, before studying the effect of two parameters in \resectwo{hete_levels}{exp_typical_times}. Finally, we show our approach still holds on another reaction model in \resec{appli_sepir}.

\subsection{Experimental Set-up}
\label{sec:exp_setup}

For each numerical simulation, we start by generating a graph from a random graph generator, specified below. Then, to obtain $\vecorate$, we draw the outrates uniformly on $[0,4\times 10^{-1}]$. The coefficients of the parameter matrix $\prm$ of the softmax reference policy are drawn uniformly in $[-10^{-1}, 10^{-1}]$. Finally, when needed, we renormalise the diffusion matrix by the typical time of diffusion $\tau >0$, which values we specify below.


The epidemiological coefficients are distributed according to the absolute values of normal variables, which parameters are available in the configuration files in the code. Having drawn the $\delta_\node$'s, we compute the $\beta_\node$'s coefficients in such a way that the basic reproduction numbers of the nodes, $\brn(\node)=\beta_\node \statiodmref(\node)/\delta_\node$, are distributed around the threshold $1$: some $\brn(\node)$'s are greater than $1$, and some lesser. 

For each loss, we write $\prmopt$ the parameter obtained at the end of training. With it, we can compute the optimal policies $\pol(\prmopt)$ associated with the different losses.

We then simulate the epidemic with the different policies on the time interval $[0, 1000]$. We use a uniform time discretisation step of $\Delta t = \min(1, \tau)$, where $\tau$ is the typical time at which diffusion occurs, and a time-discretisation scheme coinciding at first order with an Euler scheme, but preserving the positivity of the vectors $S$, $E$, $I$ and $R$. For each setting, the population is initially distributed according to the reference stationary distribution $\statiodref$, and in $2$ nodes chosen at random, $5\%$ of the population is changed from susceptible to exposed.

Finally, we measure the worth of every policy by the relative final size of the epidemic with respect to the reference policy. Namely, if the parameter of the policy is $\prmopt$, we compute $\sum_{\node\in\nods} R^{\mdpl(\prmopt)}_\node(\infty)/\sum_{\node\in\nods} R^{\mdplref}_\node(\infty)$.

\subsection{Basic Reproduction Number and Final Size}
\label{sec:influence_brn}
First, we study the relation between the basic reproduction number and the final size of the epidemic. We display on \refig{influence_brn} the final size as a function of the basic reproduction number. We conduct the experiment for four random graphs: the \erdosrenyi, which is a standard model, the Waxman graph, which is a geometric graph, and the Relaxed Caveman graph and \brbalbert, which exhibit a somewhat more constrained structure. We see the relation is increasing (the bigger the $\brn$, the bigger the final size), and that this applies to all types of graphs. That reducing the $\brn$ would ultimately reduce the final size was expected, but we also knew the relation was not straightforward (\resectwo{fsr}{fsize}). Therefore, these results validate our approach. The lines on the plot are the regression lines. We see the relation is closer to linear for the \waxman graphs (with sum of squared residuals $R^2=0.48$), than for the other graphs ($R^2 \geq 0.70$, with maximum at $R^2=0.86$ for \erdosrenyi graphs), though we do not see a clear explanation in terms of the different graphs topologies.  

\begin{figure}
    \centering
    \includegraphics[width=\linewidth]{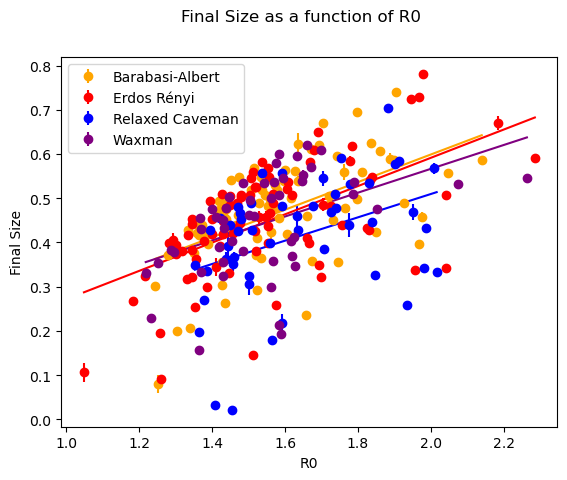}
    \caption{Final Size as a Function of the $\brn$, for Various Graphs of Size $30$, for a SEIR reaction. The lines are the regression lines.}
    \label{fig:influence_brn}
\end{figure}

\subsection{Policies Comparison}
\label{sec:direct_comparison}
Then, we compare the performances of the three policies obtained form the corresponding three losses. First, we consider standard \erdosrenyi graphs of size ranging from $10$ to $50$ nodes, and show their relative final sizes. We see on \refig{var_sizes_rel_rcv} the performance is quite good for all policies and all sizes of graphs, as the final size is reduced by at least $20\%$ for all policies. Then, we see \epidpolaccr performs best overall.
Second, in \retab{various_graphs_seir}, we report the performance of the losses, for different network topologies. We use the four random graphs already used in \resec{influence_brn}.
All graphs have size $30$. Again, we show the relative final sizes. The performances remain quite good for every type of graph and every policy, with the median consistently below the $50\%$ level, and often much below. \epidpolaccr performs best in all cases. We think it is because the loss it comes from incorporates more information about the overall dynamics of the system, than the other two losses, which are derived as limit cases. In particular, $\epidaccr$ has a direct dependency on the diffusion matrix $\mdpl(\prmopt)$, contrary to the other losses. No discernible pattern is distinguishable between the topologies, suggesting performance is not too sensitive to it. These simulations prove the overall worth of our approach for epidemic control. 

\begin{figure}
    \centering
    \includegraphics[width=\linewidth]{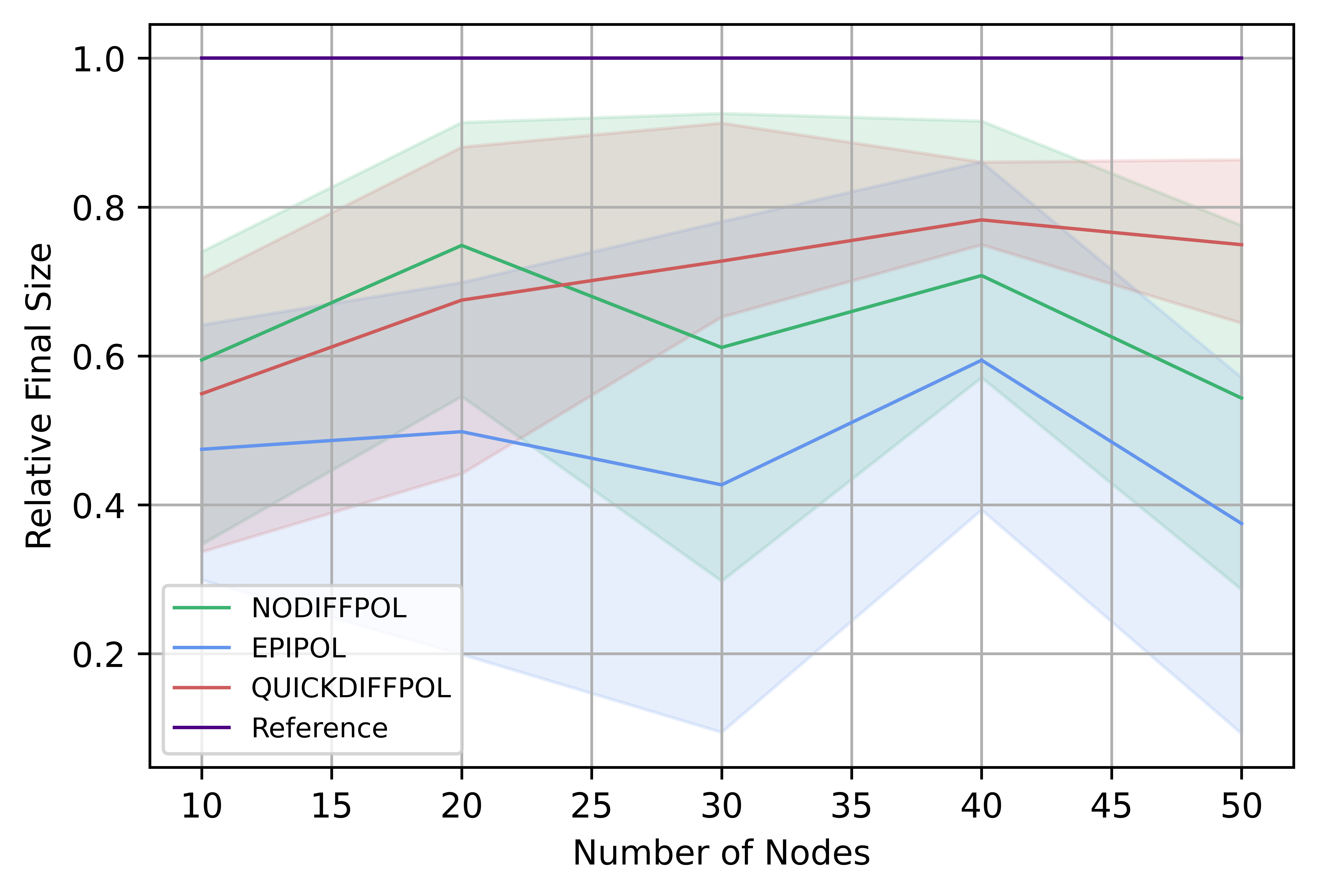}
    \caption{Relative final size for the three policies, for various sizes of \erdosrenyi graphs, for a SEIR reaction. The final sizes are those of the policies $\epidpolaccr$ (blue), $\rapsumpolaccr$ (red) and $\alphapolaccr$ (green), obtained from their respective losses. The lines are the median values, and the shaded areas gather the values in the $[25\%,75\%]$ intervals.}
    \label{fig:var_sizes_rel_rcv}
\end{figure}

\input{./tables/var_graphs_rel_rcv_seir.tex}

\subsection{Influence of the Network Heterogeneity}
\label{sec:hete_levels}

Next, we study the influence of network heterogeneity. We call heterogeneity of the network the dispersion of the values of the individual basic reproduction numbers, and of the $\delta_\node'$s. For several values of $x\in[0, 1]$, \erdosrenyi graphs of size $30$ were generated, and the individual basic reproduction numbers, and the $\delta_\node$'s, were randomly sampled from normal distributions with standard deviations equal to $x\sigma_{\delta}$ and $x \sigma_{\brn}$, respectively. The scale factor $x\in[0,1]$ thus quantifies the heterogeneity of the network. (The $\gamma_\node$'s were left constant, so as not to advantage the epidemiological loss, which could take direct advantage of it, while the limit losses could only do so indirectly, through the stationary distribution.)

\begin{figure}
    \centering
    \includegraphics[width=\linewidth]{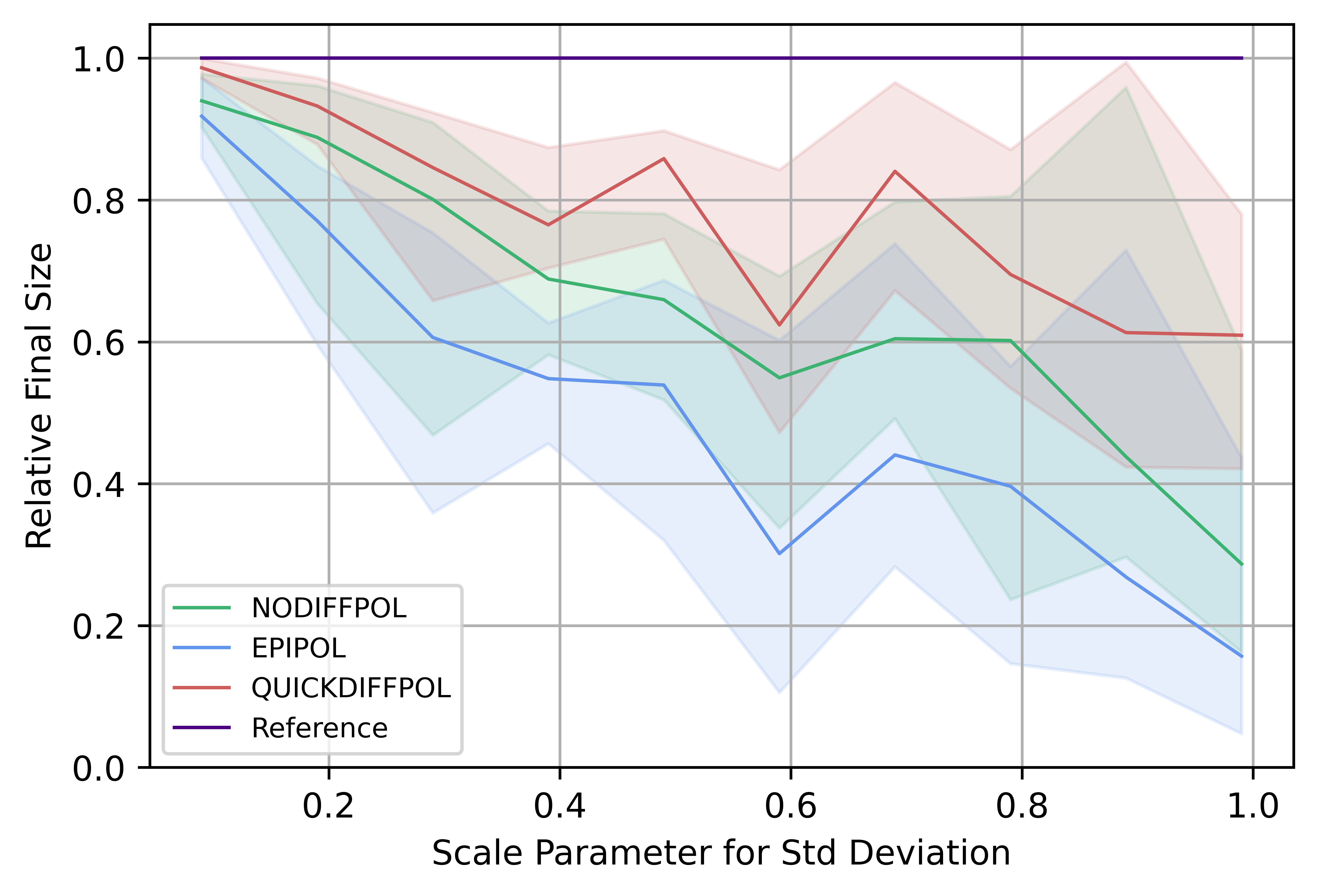}
    \caption{Relative final size for different heterogeneity levels, for a SEIR reaction, for \erdosrenyi Graphs of size $|\nods|=30$. The final sizes are those of the policies $\epidpolaccr$ (blue), $\rapsumpolaccr$ (red) and $\alphapolaccr$ (green), obtained from their respective losses. The lines are the median values, and the shaded areas gather the values in the $[25\%,75\%]$ intervals.}
    \label{fig:var_heterog_rel_rcv}
\end{figure}
 
We show on \refig{var_heterog_rel_rcv} the relative final size as a function of the heterogeneity factor $x$. We first see the final size decreases for all policies, as the scale parameter tends to $1$: indeed, the more heterogeneous the network, the more leeway there is for optimisation. Then, we see \epidpolaccr is consistently performing better than the policies derived from the limit losses: this suggests that \epidaccr is better able to exploit the heterogeneity, which we think is linked to the fact it incorporates knowledge about the dynamics, and not only the population distribution. These results show that, in order to control the epidemic by acting on the flows, there needs to be disparities in the reaction terms of the network: if all nodes \gu{look the same}, redirecting the flows will not be very worthy. On the contrary, as soon as the network displays some degree of heterogeneity, flow redirection proves efficient.

\subsection{Influence of the Rate of Diffusion}
\label{sec:exp_typical_times}

\begin{figure}
    \centering
    \includegraphics[width=\linewidth]{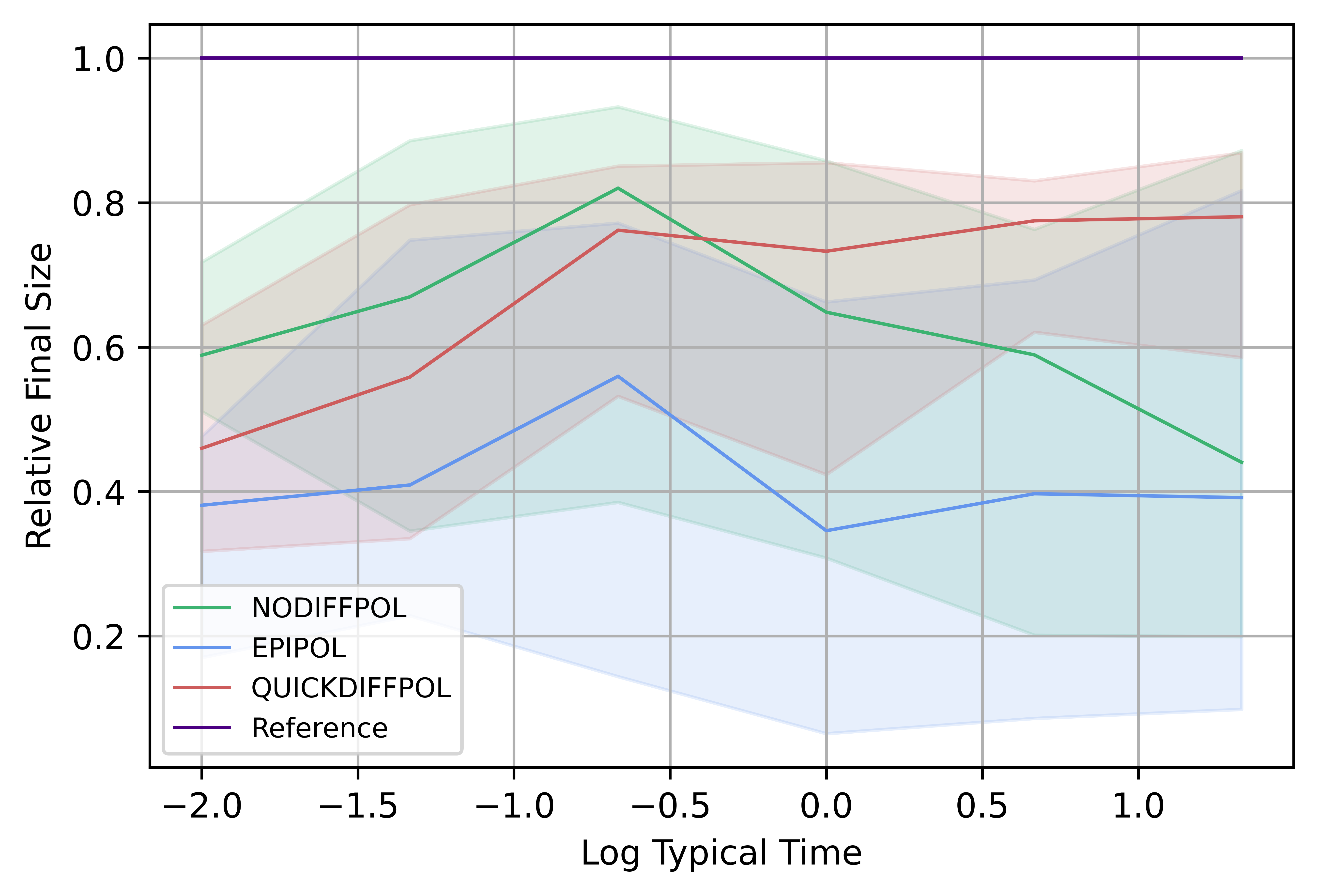}
    \caption{Relative final size for different typical times, for an SEIR reaction, for \erdosrenyi Graphs of size $\abs{\nods}=25$. The final sizes are those of the policies $\epidpolaccr$ (blue), $\rapsumpolaccr$ (red) and $\alphapolaccr$ (green), obtained from their respective losses. The lines are the median values, and the shaded areas gather the values in the $[25\%,75\%]$ intervals.}
    \label{fig:var_typical_times_rel_rcv}
\end{figure}

We now study the influence of the rate of diffusion. On \refig{var_typical_times_rel_rcv}, we show the relative final sizes for a range of typical times $\tau$ of diffusion (equivalently, $1/\tau$ is the rate of diffusion). When $\tau\to 0$, the diffusion happens very quickly, while it happens slowly when $\tau \to \infty$. First, we see that for high typical times, \epidpolaccr and \alphapolaccr, give close results: indeed, \alphalaccr is derived from the limit of the basic reproduction number when $\tau\to\infty$ (\resec{optim_flows}). Second, when $\tau$ diminishes, the results for \alphapolaccr deteriorate: \alphalaccr is no longer fit for these values. Third, the performance of \rapsumpolaccr improves when $\tau\to 0$. Again, this was expected, as \rapsumaccr is tailored for the limit case of very quick diffusion. Fourth, and finally, we see that the \epidaccr produces the best results across the range of times of diffusion, emphasising its overall worth for epidemic reaction-diffusion control. Therefore, these results highlight the importance of the rate of diffusion regarding the performance of epidemic control with flow redirection.

\subsection{Application to a SEPIR reaction}
\label{sec:appli_sepir}
Finally, to show the approach we have developed is somewhat generic with respect to the reaction model, we now show we can control a SEPIR reaction-diffusion. The SEPIR model, described in the supplementary materials, 
is the extension to a reaction-diffusion of the scalar model introduced in \citealizon \citet{alizon:hal-02882687} under the name \gu{SEAIR}. We conducted experiments on various graphs of size $\abs{\nods}=30$, and report the results in \retab{various_graphs_seir}. As is the case for SEIR, we see that performances are satisfying accross different topologies, and for every policy considered. Again, the epidemiological policy \epidaccr (which expression for SEPIR is detailed in \reapp{sepir}) gives the best policy. The relative final sizes tend to be smaller than for SEIR. This is probably linked to the fact the final size decreases quicker with the basic reproduction number than for SEIR, as we illustrate in \reapp{sepir_brn_fsize}. The results for SEPIR suggest the approach could be extended to many reaction-diffusion systems.
\input{./tables/var_graphs_rel_rcv_sepir.tex}

%% file: tables/var_graphs_rel_rcv_seir.tex
\begin{table}
\centering
\scriptsize
\begin{tabular}{ |p{0.2cm}|p{1.65cm}|p{1.65cm}|p{1.65cm}|p{1.65cm}|  }
\hline
\multicolumn{1}{|c}{} & \multicolumn{4}{|c|}{Graph Type} \\
\hline
Pol. & \erdosrenyi & \waxman & \brbalbert & \relcav \\
\hline
(A) & \textbf{0.55} [0.38, 0.70] & \textbf{0.26} [0.14, 0.50]  & \textbf{0.29} [0.01, 0.52] & \textbf{0.34} [0.05, 0.53] \\
(B) & 0.63 [0.44, 0.74] & 0.59 [0.31, 0.79] & 0.44 [0.15, 0.69] & 0.40 [0.05, 0.68] \\
(C) & 0.71 [0.70, 0.92] & 0.55 [0.57, 0.74] & 0.59 [0.17, 0.83] & 0.60 [0.37, 0.70] \\
\hline
\end{tabular}
\caption{Relative Final Size for Various Policies (obtained from the corresponding losses) and Various Graphs of size 40 for an SEIR reaction. Median, and [20\%, 80\%] intervals indicated. In bold, the smallest value for each graph type. (A): \epidpolaccr; (B): \rapsumpolaccr; (C): \alphapolaccr.}
\label{tab:various_graphs_seir}
\end{table}

%% file: tables/var_graphs_rel_rcv_sepir.tex
\begin{table}
\centering
\scriptsize
\begin{tabular}{ |p{0.2cm}|p{1.65cm}|p{1.65cm}|p{1.65cm}|p{1.65cm}|  }
\hline
\multicolumn{1}{|c}{} & \multicolumn{4}{|c|}{Graph Type} \\
\hline
Pol. & \erdosrenyi & \waxman & \brbalbert & \relcav \\
\hline
(A) & \textbf{0.39} [0.11, 0.76] & \textbf{0.19} [0.09, 0.53]  & \textbf{0.19} [0.04, 0.57] & \textbf{0.14} [0.03, 0.40] \\
(B) & 0.48 [0.17, 0.70] & 0.56 [0.17, 0.70]  & \textbf{0.19} [0.06, 0.58]  & 0.33 [0.08, 0.74] \\
(C) & 0.46 [0.17, 0.72] & 0.42 [0.16, 0.72]  & 0.26 [0.11, 0.59] & 0.22 [0.05, 0.57] \\
\hline
\end{tabular}
\caption{Relative Final Size for Various Policies (obtained from the corresponding losses) and Various Graphs of size $30$ for an SEPIR reaction. Median, and [20\%, 80\%] intervals indicated. In bold, the smallest value for each graph type. (A): \epidpolaccr; (B): \rapsumpolaccr; (C): \alphapolaccr.}
\label{tab:various_graphs_sepir}
\end{table}

%% file: sources/conclusion.tex
We have shown we can control an epidemic reaction-diffusion on a directed, and heterogeneous, network by redirecting the flows, thanks to the optimisation of well-designed loss functions, in particular the basic reproduction number of the model. We have provided a final size relation linking the basic reproduction number to the epidemic final sizes, for diffusions around a reference diffusion with basic reproduction number less than $1$. Numerically, we have shown control is possible for different topologies, network heterogeneity levels, and speeds of diffusion. Moreover, our experimental results highlight the relevance of the $\brn$-based loss, compared to more straightforward losses. However, these improved performances should be balanced against the highest computational costs it entails with respect to the other losses.

Overall, we believe our results make the case for flow redirection as a relevant control tool of epidemic reaction-diffusion systems. Further, we have identified key network parameters which may usefully inform the optimisation design. In turns, this stresses the need for quality data collection about networks.

One theoretical limitation of our work is the fact the final size relation only stands for $\brn(\mdplref)<1$: it would be interesting to extend it to the case $\brn(\mdplref)>1$. Then, optimisation reduces the final size, but modifies the network flow structure. Attempting to control the dynamics, while modifying as little as possible valuable metrics representing the usual flow structure of the network, would represent an interesting direction of future research.

%% file: appendices/proof_unif_stability.tex
To prove \reprop{fsize-r0-smaller-one}, we first need a strengthening of the standard stability result recalled in \resec{metapop_models}. This strenghtening is obtained in \relem{unif-stab-dfe}, which establishes a uniform stability property for the Disease Free equilibrium. 
\begin{lemma}[Uniform stability of the Disease Free Equilibrium]
\label{lem:unif-stab-dfe}
Let $\mdplref$ be a diffusion matrix, and assume $\brn(\mdplref)<1$. Then, for any $\eps >0$, there is a ball $\bo$ around $\mdplref$, and $\eta > 0$ such that, for every initial condition $\big(S(0), E(0), I(0), R(0)\big)$ satisfying
$\nrm{\big(S(0), E(0), I(0), R(0)\big) - \big(\statiodmref, 0, 0, 0\big)} < \eta$,
for any diffusion matrix $\mdpl$ in $\bo$, for all $t\geq 0$, we have
\begin{equation*}
\nrm{\paren{S^\mdpl(t), E^\mdpl(t), I^\mdpl(t), R^\mdpl(t)} - \paren{\statiodfmat{\mdpl}, 0, 0, 0}}\leq \eps,
\end{equation*}
where $\big(S^\mdpl(t), E^\mdpl(t), I^\mdpl(t), R^\mdpl(t)\big)$ is the solution for $t \geq 0$ of \reeq{seir-metapop}, with initial condition $\big(S(0), E(0), I(0), R(0)\big)$.
\end{lemma}

To prove \relem{unif-stab-dfe}, we need several intermediary results.

\begin{lemma}[Uniform exponential boundedness]
\label{lem:cont_fdtal_sol}

Let $\mathbf{Q}_0$ be a, finite dimensional, matrix the eigenvalues of which have all negative real parts. Let $\nrm{\cdot}$ be any norm on the space of matrices. Then, we may find $\lambda >0$, $\kappa\geq 0$, and a ball $\bo$ around $\mathbf{Q}_0$ such that, for every $\mathbf{Q}\in\bo$, for all $t \geq 0$, we have
$\nrm{\exp\paren{t\mathbf{Q}}} \leq \kappa\,\exp\paren{-\lambda t}$.
\end{lemma}
\begin{proof}
Let $\mathbf{P}$ be positive definite such that $\mathbf{P}\mathbf{Q}_0 + \mathbf{Q}_0^T\mathbf{P} = -\id$. Such a matrix exists because the eigenvalues of $\mathbf{Q}_0$ have negative real parts. 
For any vector $X$, we write $\nrm{X}^2_{\mathbf{P}} = X^T \mathbf{P} X$. Since all norms are equivalent, we can choose $\alpha > 0$ such that $\alpha\nrm{X}_{\mathbf{P}} \leq \nrm{X}$ for all $X$.
Let then $\mathcal{B}$ be a ball around $\mathbf{Q}_0$, such that, for every $\mathbf{Q}\in\mathcal{B}$, 
all the eigenvalues of $\mathbf{P}\mathbf{Q} + \mathbf{Q}^T\mathbf{P}$ are negative. This is possible because $\mathbf{Q}\mapsto \mathbf{P}\mathbf{Q} + \mathbf{Q}^T \mathbf{P}$ is continuous, and because the matrices $\mathbf{P}\mathbf{Q}+\mathbf{Q}^T \mathbf{P}$ are symmetric, so that all their eigenvalues are real. Let $-\lambda<0$ be any upper bound for the eigenvalues of the $\mathbf{P}\mathbf{Q}+\mathbf{Q}^T\mathbf{P}$'s for $\mathbf{Q}\in\mathcal{B}$.
Let now $\mathbf{Q} \in \mathcal{B}$, $X_0$ be a vector, and $X$ be the solution of
$\frac{dX}{dt} = \mathbf{Q} \, X$,
with initial condition $X(0)=X_0$.
For all $t\geq 0$, we have
\begin{align*}
\frac{d}{dt} \nrm{X}_{\mathbf{P}}^2  
&= X^T \paren{\mathbf{Q}^T \mathbf{P} + \mathbf{P} \mathbf{Q}} X \\
&\leq -\lambda \, \nrm{X}^2 
\leq -\alpha \lambda \nrm{X}_{\mathbf{P}}^2,
\end{align*}
by construction of $\lambda$. As a consequence, for all $t\geq 0$, by comparison, we known that $\nrm{X(t)}_{\mathbf{P}} \leq \exp\paren{-\frac{\alpha \lambda}{2} t} \nrm{X_0}_{\mathbf{P}}$, and therefore $\nrm{\exp\paren{\mathbf{Q}t}X_0}_{\mathbf{P}}\leq \exp\paren{-\frac{\alpha \lambda}{2} t} \nrm{X_0}_{\mathbf{P}}$. Since $X_0$ was any vector, the operator norm $\nrmop{\cdot}$ associated with $\nrm{\cdot}_{\mathbf{P}}$ at the start and finish satisfies, for all $t\geq 0$, $\nrmop{\exp\paren{\mathbf{Q}t}} \leq \exp\paren{-\frac{\alpha \lambda}{2} t}$.
Being in finite dimension, all norms are equivalent so that, for some $\beta \geq 0$ independent of $\mathbf{Q}$, we have, for all $t\geq 0$, $\nrm{\exp\paren{\mathbf{Q}t}}\leq \beta\nrmop{\exp\paren{\mathbf{Q}t}}$. As result, for all $\mathbf{Q}\in\mathcal{B}$, for all $t\geq 0$, we have $\nrm{\exp\paren{\mathbf{Q}t}}\leq \beta\exp\paren{-\frac{\alpha \lambda}{2} t}$.
\end{proof}

\begin{lemma}[Uniform closeness to the stationary distribution]
\label{lem:cont_sols_chapman}

Let $\mdplref$ be a diffusion matrix.
For all $\eps>0$, we can find a ball $\mathcal{B}$ around $\mdplref$, and $\eta >0$, such that, for all initial distribution $N_0$ verifying $\abs{N_0-\statiodfmat{\mdplref}}<\eta$, for all diffusion matrix $\mdpl \in \mathcal{B}$, for all $t\geq 0$, we have
$\nrm{N_\mdpl(t) - \statiodfmat{\mdpl}} < \eps$,
where $N_\mdpl(t) = \exp\paren{t\mdpl}N_0$ is the solution starting at $N_0$ of $dN/dt=\mdpl N$.
\end{lemma}

\begin{proof}
Let $\mdpl$ be a diffusion matrix. We know that
$\mathbb{R}^{\abs{\nods}} = \mathbb{R}\,\statiodfmat{\mdpl} \oplus \mathcal{H}$, where $\mathcal{H}=\enstq{\nu \in \mathbb{R}^{\abs{\nods}}}{\sum_{\node} \nu_\node = 0}$.
Let $N_0$ be an initial distribution.
We decompose
$N_0 = \alpha_\mdpl \, \statiodfmat{\mdpl} + \nu$, with $\alpha_\mdpl$ a real, and $\nu\in \mathcal{H}$.   
Indeed, $\alpha_\mdpl = \sum_{n\in\nods} N_{0,n} = 1$ since $N_0$ is a distribution.
For all $t \geq 0$, we have
$N_\mdpl \paren{t} - \statiodfmat{\mdpl} = \exp \paren{t\mdpl}N_0 - \statiodfmat{\mdpl}
= \exp\paren{t\mdpl} \paren{\statiodfmat{\mdpl} + \nu} - \statiodfmat{\mdpl}
= \exp\paren{t\mdpl}\statiodfmat{\mdpl} - \statiodfmat{\mdpl} + \exp\paren{t\mdpl}\nu
= \exp\paren{t\mdpl}\nu$,
since $\mdpl\statiodfmat{\mdpl}=0$.
Therefore, for all $t\geq 0$, we have
$\nrm{N_\mdpl \paren{t} - \statiodfmat{\mdpl}} \leq \nrmop{\exp \paren{t\mdpl}} \nrm{\nu}$.
Now, we also have 
$\nrm{\nu} = \nrm{N_0 - \Tilde{\mu}_\mdpl}
= \nrm{N_0 - \statiod + \statiod - \Tilde{\mu}_\mdpl}
 \leq \nrm{N_0 - \statiod} + \nrm{\statiod - \Tilde{\mu}_\mdpl}$.
All diffusion matrices have only eigenvalues with negative real parts in $\mathcal{H}$, so that it is true in particular for $\mdplref$ and, thanks to the proof of \relem{cont_fdtal_sol}, we may find a ball $\bo$ of diffusion matrices around $\mdplref$, and $\kappa \geq 0$ such that, for every $\mdpl\in\mathcal{B}$, we have, for all $t\geq 0$,
$\nrm{\exp\paren{t\mdpl}}_{\mathrm{op},\mathcal{H}}<\kappa$,
where $\nrm{\cdot}_{\mathrm{op},\mathcal{H}}$ is the operator norm for the restriction of matrices $\mdpl$ to the space $\mathcal{H}$. We used the fact the restriction to $\mathcal{H}$ is a continuous function of the matrix.

Let then $\eps>0$. Now, the stationary distribution of a diffusion matrix depends continuously on the matrix \cite{golubmey86}.
Therefore, upon diminishing $\bo$, for all diffusion matrix $\mdpl\in\bo$, we may assume $\nrm{\statiodmref - \Tilde{\mu}_\mdpl} \leq \varepsilon/\kappa$. Choose $N_0$ such that $\nrm{N_0-\statiodfmat{\mdplref}}< \eps/\kappa$. As a result, thanks to \relem{cont_fdtal_sol}, for $\mdpl \in \mathcal{B}$, for all $t\geq 0$, we obtain 
$\nrm{N_\mdpl \paren{t} - \mu_\mdpl}  \leq \nrm{\exp\paren{t\mdpl}}_{\mathrm{op},\mathcal{H}} \frac{\eps + \eps}{\kappa}
\leq  2\,\eps$.
\end{proof}

We now introduce the following notations.
For every diffusion matrix $\mdpl$, and every initial condition, let us define (letting the dependency of $E_+$ and $I_+$ on $\mdpl$ be implicit so as to simplify notations),
\begin{equation*}
\left\lbrace \ba
\frac{d E_+}{dt} &= \diagbeta N_\mdpl \odot I_+ + \paren{\mdpl-\diaggamma}\,E_+\\
\frac{d I_+}{dt} &= \diaggamma\,E_+ + \paren{\mdpl-\diagdelta} I_+.
\ea \right.
\end{equation*}
For every $\mdpl$, and every $t\geq 0$, define further
\begin{equation*}
\mathbf{A}(\mdpl,t)
=
\begin{pmatrix}
\mdpl-\diaggamma & \diagbeta \mathrm{Diag}\paren{N_\mdpl(t)} \\
\diaggamma & \mdpl-\diagdelta
\end{pmatrix}.
\end{equation*}
Then, in matrix notations, we have
$\cpl{\frac{d E_+}{dt}}{\frac{dI_+}{dt}}
=
\mathbf{A}(\mdpl,t)
\cpl{E_+(t)}{I_+(t)}.$

\begin{lemma}[Upper-bounding linear system]
\label{lem:upblinear-inf}
Let $\mdplref$ be a diffusion matrix. Define $\mathbf{A}=\mathbf{A}(\mdplref,\infty)$, that is
$\mathbf{A}
=
\begin{pmatrix}
\mdplref-\diaggamma & \diagbeta \mathrm{diag}\paren{\statiodmref} \\
\diaggamma & \mdplref-\diagdelta
\end{pmatrix}$.
Assume $\brn(\mdplref)<1$. 
Then, we can find a ball $\mathcal{B}$ around $\mdplref$, $\lambda > 0$, $\eta >0$, and $\kappa\geq 0$ such that, for all $\nrm{N_0-\statiodmref}<\eta$, for all diffusion matrices $\mdpl \in \mathcal{B}$, for all $t\geq 0$, we have
$\nrm{\cpl{E_+(t)}{I_+(t)}} \leq \kappa\,\exp \paren{-\lambda t} \nrm{\cpl{E_+(0)}{I_+(0)}}$.
\end{lemma}

\begin{proof}
Let us start by proving that the $\mathbf{A}(\mdpl,t)$'s are uniformly close to $\mathbf{A}$, subject to some conditions we now precise. Let $\eps>0$.
Thanks to \relem{cont_sols_chapman}, we can find  a ball $\mathcal{B}$ around $\mdplref$, and $\eta >0$, such that, for all diffusion matrix $\mdpl \in \mathcal{B}$, for all $\nrm{N_0-\statiodmref}<\eta$, for all $t\geq 0$, we have
$\nrm{N_\mdpl(t) - \Tilde{\mu}_\mdpl} \leq \eps$.
Now, upon diminishing $\bo$, we may also assume that, for all diffusion matrix  $\mdpl\in\bo$, $\nrm{\statiodmref-\statiodfmat{\mdpl}}<\eps$. As a result, for all  such matrix $\mdpl\in\bo$, for all $\nrm{N_0-\statiodmref}<\eta$, for all $t\geq 0$, we have
$\nrm{N_\mdpl(t) - \statiodmref} \leq \nrm{N_\mdpl(t) - \Tilde{\mu}_\mdpl} + \nrm{\Tilde{\mu}_\mdpl - \statiodmref} \\
\leq \eps+\eps
\leq 2\,\varepsilon$.
Now, all the other coefficients of $\mathbf{A}(\mdpl,t)$ depend continuously on $\mdpl$ so that, upon diminishing $\bo$ further, we may assume that, for all diffusion matrix  $\mdpl\in\bo$, for all  $\nrm{N_0-\statiodmref}<\eta$, for all $t\geq 0$, we have
$\nrm{\mathbf{A}(\mdpl,t)-\mathbf{A}}_{\infty}<2\,\varepsilon$,
where $\nrm{\cdot}_\infty$ is the infinity norm on the coefficients of the matrix, and we conclude by invoking the equivalence of norms.

Now, $\brn(\mdplref)$ is strictly less than one so that, thanks to \citet{MAT09}, 
all the eigenvalues of $\mathbf{A}$ have negative real parts.
Proceeding as in the proof of \relem{cont_fdtal_sol}, we may find $-1 < - \lambda <0$, and a neighbourhood of $\mathbf{A}$ such that, for every matrix $\mathbf{B}$ inside it, all the eigenvalues of $\mathbf{P}\mathbf{B}+\mathbf{B}^T \mathbf{P}$ are real and strictly less than  $-\lambda$.
By what precedes, upon choosing $\bo$ and $\eta$ small enough, we have that, for every diffusion matrix $\mdpl\in\bo$, for every $\nrm{N_0-\statiodmref}<\eta$, for every $t\geq 0$, all the eigenvalues of $\mathbf{P}A(\mdpl,t)-A(\mdpl,t)^T \mathbf{P}$ are inferior to $-\lambda$.

Fix now a diffusion matrix $\mdpl\in\bo$, and $\nrm{N_0-\statiodfmat{\mdpl}}<\eta$. Let then $X_0$ be a vector, and $X=\cpl{E_+}{I_+}$ be the solution of
$\frac{dX}{dt} = A(\mdpl,t) \, X$
with $X(0)=X_0$. Again, as in the proof of \relem{cont_fdtal_sol}, we obtain some $\alpha \geq 0$ independent of $\mdpl$ such that, for all $t\geq 0$, we have
$\nrm{X(t)}_P \leq Y(t)^{1/2} \leq \exp \paren{\frac{-\lambda \alpha}{2} t} \nrm{X_0}_P$,
where $Y$ is the solution, for $t \geq 0$, of $\frac{dY}{dt} = -\lambda \alpha Y$, with initial condition $Y(0) = \nrm{X_0}^2$.

Thanks to the equivalence of norms, we may find $\kappa \geq 0$ such that, for all $t\geq 0$, we have
$\nrm{X(t)} \leq \kappa\,\exp \paren{\frac{-\lambda \alpha}{2} t} \nrm{X_0}$.
This stands for any diffusion matrix $\mdpl\in\bo$, and any $\nrm{N_0-\statiodmref}<\eta$, so that we have proven our claim.
\end{proof}

\begin{lemma}[Comparison] 
\label{lem:upbounding-inf-comp}
Assume $\brn(\mdplref)<1$.
There exists a ball $\bo$ around $\mdplref$, $\lambda>0$, $\eta>0$ and $\kappa\geq 0$ such that, for any diffusion matrix $\mdpl\in\bo$, for all initial distribution $\nrm{N_0-\statiodmref}<\eta$, for all $t\geq 0$, we have
$\cpl{E^\mdpl(t)}{I^\mdpl(t)}
\leq
\cpl{E^\mdpl_+(t)}{I^\mdpl_+(t)}$,
and $\nrm{\cpl{E^\mdpl(t)}{I^\mdpl(t)}} \leq \kappa\,\exp \paren{-\lambda t} \nrm{\cpl{E(0)}{I(0)}}$,
where $\cpl{E^\mdpl(t)}{I^\mdpl(t)}$ are the $E$ and $I$ coordinates of the system of \reeq{seir-metapop} when the diffusion matrix is $\mdpl$, and the initial population is distributed according to $S(0)+E(0)+I(0)+R(0)=N_0$, and $\cpl{E_+^\mdpl}{I_+^\mdpl}$ are introduced before \relem{upblinear-inf}, and have initial condition $\cpl{E(0)}{I(0)}$. 
\end{lemma}

\begin{proof}
Choose $\bo$, $\lambda$, $\eta$ and $\kappa$ as in \relem{upblinear-inf}. 
Fix an initial condition $\paren{S(0), E(0), I(0), R(0)}$ such that $\nrm{N_0-\statiodmref}<\eta$. Fix a diffusion matrix $\mdpl\in\bo$. We drop the \gu{$\mdpl$ exponents} to simplify the notations. 
By definition, we know that, for all $t\geq 0$, we have
$\frac{d E}{dt} = \diagbeta S \odot I-\diaggamma\,E+\mdpl\,E$, and
$\frac{d I}{dt} = \diaggamma\,E-\diagdelta I + \mdpl I$.
Now, for all $t\geq 0$, we know that $S(t) \leq N_\mdpl(t) = S(t) + E(t) + I(t) + R(t)$.
As a result, for all $t\geq 0$, we have $\frac{d E}{dt} \leq \diagbeta N_\mdpl(t) \odot I-\diaggamma\,E+\mdpl\,E$,
so that
$\cpl{\frac{d E}{dt}}{\frac{dI}{dt}}
\leq
\mathbf{A}(\mdpl,t)
\cpl{E(t)}{I(t)}$.
Now, for all $t\geq 0$, $\mathbf{A}(\mdpl,t)$ is a Metzler matrix, so that, thanks to Section 5.5 of \cite{sallet18}, $f(t,X) = \mathbf{A}(\mdpl,t)X$ is of type $K$. As a consequence, we may use the comparison Theorem B.1 of \citet{smith95} to obtain that, for all $t\geq 0$,
$\cpl{E(t)}{I(t)}
\leq
\cpl{E_+(t)}{I_+(t)}$.
\end{proof}

We can now prove \relem{unif-stab-dfe}. 


\begin{proof}
Thanks to \relem{cont_sols_chapman} and \relem{upbounding-inf-comp}, there exists a ball $\bo$ around $\mdplref$, $\lambda>0$, $\eta>0$ and $\kappa\geq 0$ such that, for any diffusion matrix $\mdpl\in\bo$, for all initial distribution $\nrm{N_0-\statiodmref}<\eta$, for all $t\geq 0$, we have
$\nrm{N_\mdpl(t)-\statiodfmat{\mdpl}} < \eps$,
and
$\nrm{\cpl{E^\mdpl(t)}{I^\mdpl(t)}} \leq \kappa\,\exp \paren{-\lambda t} \nrm{\cpl{E(0)}{I(0)}}$.
Fix a diffusion matrix $\mdpl\in\bo$ and some initial condition satisfying the requirements above.
Note that by assumption $\nrm{\cpl{E(0)}{I(0)}} \leq \eta$.
Hence, for all $t\geq 0$, we have
$\nrm{\cpl{E^\mdpl(t)}{I^\mdpl(t)}} \leq \kappa\,\eta\,\exp \paren{-\lambda t}$.
As a result, for all $t\geq 0$, we have
\begin{align*}
&\nrm{\paren{S^\mdpl(t), E^\mdpl(t), I^\mdpl(t), R^\mdpl(t)} - \paren{\statiodfmat{\mdpl}, 0, 0, 0}}  \\
\leq &\nrm{S^\mdpl(t)-\statiodfmat{\mdpl}} + \nrm{E^\mdpl(t)} + \nrm{I^\mdpl(t)} + \nrm{R^\mdpl(t)} \\
\leq &\nrm{N_\mdpl(t) - \statiodfmat{\mdpl}} + 2\nrm{E^\mdpl(t)} + 2\nrm{I^\mdpl(t)} + 2\nrm{R^\mdpl(t)} \\
\leq & \quad \eps + 4 \, \kappa \, \eta + 2\nrm{R^\mdpl(t)}
\end{align*}
since, for all $t\geq 0$, $S^\mdpl(t)+E^\mdpl(t)+I^\mdpl(t)+R^\mdpl(t) = N_\mdpl(t)$. 
Now, for all $t\geq 0$, $\nrm{R^\mdpl(t)} \leq \sum_\node R_\node^\mdpl(t)\leq \sum_\node R_\node^\mdpl(\infty)$, as the individuals who arrive in some compartment $R$ stay there indefinitely.
Now, for all $t\geq 0$, we have
$\sum_{\node\in\nods} \frac{d R_\node(t)}{dt} = \sum_{\node\in\nods} \delta_\node I_\node(t)$,
so that, for all $t\geq 0$, we have
\begin{multline*}
\sum_{\node\in\nods}R_\node(\infty) - \sum_{\node\in\nods} R_\node(0) =
\sum_{\node\in\nods} \delta_\node \int_0^\infty I_\node(t) dt \\
\leq \kappa\,\eta \, \sum_{\node\in\nods} \delta_\node \int_0^\infty \exp\paren{-\lambda t} dt
= \frac{\kappa\,\eta }{\lambda} \, \sum_{\node\in\nods} \delta_n.
\end{multline*}
As a result, for all $t\geq 0$, remembering $\sum_{\node\in\nods}R_\node(0) \leq \eta N$, we have
\begin{multline*}
\nrm{\paren{S^\mdpl(t), E^\mdpl(t), I^\mdpl(t), R^\mdpl(t)} - \paren{\statiodfmat{\mdpl}, 0, 0, 0}} 
\\ \leq
\eps + 4 \, \kappa \, \eta + 2\eta N + 2\frac{\kappa\,\eta }{\lambda} \, \sum_{\node\in\nods} \delta_n.
\end{multline*}
Upon diminishing $\eta$, 
we have therefore proven our claim.
\end{proof}

%% file: appendices/SEPIR_model.tex
\subsection{Model}

The SEPIR model is the extension to a reaction-diffusion of the scalar model introduced in \citet{alizon:hal-02882687} under the name \gu{SEAIR}. However, the \gu{A} for \gu{Asymptomatic} compartment in the reference behaves in fact like a \gu{Pre-symptomatic} compartment, hence renaming it \gu{P}.  The difference with SEIR is that it comprises two infectious stages, with different levels of infectiousness (the \gu{$\diagbeta$} factors are different). Compared to the SEIR metapopulation model defined in \reeq{seir-metapop},
one compartment is added, the $P$ for \gu{Pre-symptomatic} compartment, and two additional matrices of parameters are needed. The diagonal $\diagbeta^P$ matrix describes the infection of susceptible individuals by pre-symptomatic individuals, while the diagonal $\diagalpha$ matrix describes the rate at which individuals leave the pre-symptomatic compartment - the $\diagbeta$ matrix of SEIR is named $\diagbeta^I$ here. All diagonal coefficients of $\diagbeta^P$ and $\diagalpha$ are positive.
The population on the graph thus follows the dynamics 
\begin{equation*}
\label{eq:seair_graphe_vec}
\left\lbrace \ba
\frac{d S}{dt} &= -\diagbeta^I S \odot I -\diagbeta^P S \odot P + \mdpl S \\
\frac{d E}{dt} &= \diagbeta^I S \odot I + \diagbeta^P S \odot P -\diagalpha\,E+\mdpl\,E\\
\frac{d P}{dt} &= \diagalpha \,E-\diaggamma P + \mdpl P\\
\frac{d I}{dt} &= \diaggamma\,P-\diagdelta I + \mdpl I\\
\frac{d R}{dt} &= \diagdelta I + \mdpl R.
\ea \right.
\end{equation*}

Then, we computed the next-generation matrix of the model, as 
\begin{multline*}
\ngm_{\mdpl} = \diagbeta^P \diag{\statiod}\paren{\id-\paren{\diagbeta^P}^{-1}\diagbeta^I\,\paren{\mdpl-\diagdelta}^{-1}\,\diaggamma} \\
\times \paren{\mdpl -\diaggamma}^{-1}\,\diagalpha\,\paren{\mdpl -\diagalpha}^{-1}.
\end{multline*}
We were able to check it is positive. Indeed, the only factor which is not of the type of those for the SEIR model is $\id-\paren{\diagbeta^P}^{-1}\diagbeta^I\,\paren{\mdpl-\diagdelta}^{-1}\,\diaggamma$. Now, as in \relem{diff_statiod_brn},
$-\paren{\mdpl-\diagdelta}^{-1}$ is positive, therefore, since $\diagbeta^P$, $\diagbeta^I$ and $\diaggamma$ have positive diagonal coefficients, $-\paren{\diagbeta ^P}^{-1}\diagbeta^I\,\paren{\mdpl-\diagdelta}^{-1}\,\diaggamma$ is also positive, and it remains true when adding $\id$. 
Therefore, the basic reprodution number is differentiable for the SEPIR model as well, which justifies we can optimise the policies for this model as we did in the case of SEIR.

We now give explicit expressions for the losses defined in \resec{losses_control_pol},
for the SEPIR model. The first one is the epidemic loss, \epidaccr, defined by $\epidloss(\prm)=\rho\paren{ \mathbf{G}(\theta)}$.
The second loss, \alphalaccr, is the limit, when $\tau \to \infty$, of the basic reproduction number of the system of \reeq{seir-metapop},
when the diffusion is replaced by $\mdpl/\tau$. It is defined by, for every $\prm$,
\begin{equation*}
\alphaloss (\prm) = \mathcal{S}_{a} \biggl( (\diagbeta^I\diagdelta^{-1} + \diagbeta^P\diaggamma^{-1})\statiodprm \biggl),
\end{equation*}
where $\mathcal{S}_{a}: \mathbb{R}^{|\mathcal{N}|} \to \mathbb{R}$, is the $a$-smooth max function, as defined in \resec{losses_control_pol}.

Conversely, the third loss, \rapsumaccr, is the limit, when $\tau\to 0$, of the basic reproduction number, defined by, for every parameter $\prm$, 
\begin{equation*}
\rapsumloss (\prm) = \frac{\sum_\node \beta_\node^I \statiodprm(\node)^2}{\sum_\node \delta_\node \statiodprm(\node)} + \frac{\sum_\node \beta_\node^P \statiodprm(\node)^2}{\sum_\node \gamma_\node \statiodprm(\node)}.
\end{equation*}

\subsection{Basic Reproduction Number and Final Size}
\label{app:sepir_brn_fsize}
\begin{figure}
    \centering
    \includegraphics[width=.9\linewidth]{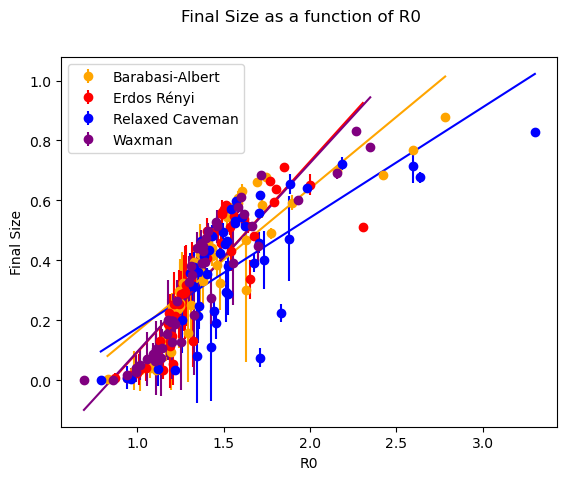}
    \caption{Final Size as a Function of the $\brn$, for Various Graphs of Size $30$, for a SEPIR reaction. The lines are the regression lines.}
    \label{fig:influence_brn_seair}
\end{figure}
On \refig{influence_brn_seair},  we display the final size as a function of the basic reproduction number, for four random graphs, for the SEPIR model. The lines on the plot are the regression lines. We see the final size diminishes when the basic reproduction number diminishes, and quicker than for SEIR.